\documentclass[final,1p,times]{elsarticle}

\usepackage{lineno}
\usepackage{amsmath}
\usepackage{amssymb}
\usepackage{amsfonts}
\usepackage{mathrsfs}
\usepackage{enumitem}
\usepackage{multirow}
\usepackage{booktabs}
\usepackage{float}
\usepackage{amsopn}
\usepackage{microtype}

\newcommand{\bx}{\boldsymbol{x}}
\newcommand{\by}{\boldsymbol{y}}
\newcommand{\bn}{\boldsymbol{n}}
\newcommand{\bt}{\boldsymbol{t}}
\newcommand{\be}{\boldsymbol{e}}
\newcommand{\bv}{\boldsymbol{v}}
\newcommand{\bu}{\boldsymbol{u}}
\newcommand{\bA}{\mathbf{A}}

\newcommand{\bK}{\mathbf{K}}

\newcommand{\bI}{\mathbf{I}}

\newtheorem{theorem}{Theorem}[section]
\newtheorem{lemma}[theorem]{Lemma}
\newtheorem{proposition}[theorem]{Proposition}

\newdefinition{definition}[theorem]{Definition}
\newdefinition{remark}[theorem]{Remark}
\newproof{proof}{Proof}
\usepackage{etoolbox}
\AtEndEnvironment{proof}{\qed}

\journal{Journal of Computational Physics}

\begin{document}

\begin{frontmatter}

\title{Singularity Analysis of Laplace Boundary Integral Equations
on Three-Dimensional Domains with Corners}

\author[amss]{Wenmao Hua}
\ead{huawenmao@amss.ac.cn}

\author[zju,cam]{Jun Lai}
\ead{laijun6@zju.edu.cn}

\author[zju]{Baoling Xie\corref{cor1}}
\ead{baolingxie@zju.edu.cn}
\cortext[cor1]{Corresponding author}

\author[ntu]{Lu Zhang}
\ead{luzhang@ntu.edu.sg}

\affiliation[amss]{organization={Academy of Mathematics and Systems Science,
            Chinese Academy of Sciences},
            city={Beijing},
            postcode={100190},
            country={China}}

\affiliation[zju]{organization={School of Mathematical Sciences,
            Zhejiang University},
            city={Hangzhou},
            postcode={310027},
            state={Zhejiang},
            country={China}}

\affiliation[cam]{organization={Center for Interdisciplinary Applied
            Mathematics, Zhejiang University},
            city={Hangzhou},
            postcode={310027},
            state={Zhejiang},
            country={China}}

\affiliation[ntu]{organization={Division of Mathematical Sciences,
            School of Physical and Mathematical Sciences,
            Nanyang Technological University},
            postcode={637371},
            country={Singapore}}

\begin{abstract}
Solutions to elliptic boundary value problems on three-dimensional nonsmooth domains exhibit singular behavior near edges and vertices. Using layer potential theory, we investigate boundary integral equations for the Laplace equation and characterize the leading singular exponents of the surface density on conical domains. For two canonical geometries, we derive explicit formulas for the two families of edge singular exponents on a dihedral wedge, and characterize the modewise vertex singular exponents on an axisymmetric circular cone through the zeros of associated Legendre functions and their derivatives. Building on these two prototypes and the dilation-invariance property, we formulate the associated operator pencil for general conical domains and develop a nonlinear eigenvalue solver. This solver employs the Nystr\"om method with graded meshes, logarithmically corrected trapezoidal rules, and Brent's method to compute these singular exponents accurately and efficiently. Numerical examples for trihedral cones, regular $m$-hedral cones, and droplet-type cones verify the effectiveness of our approach and illustrate the interactions between edge and vertex singularities.
\end{abstract}


\end{frontmatter}


\section{Introduction}
Elliptic boundary value problems on nonsmooth domains exhibit singular behavior near geometric features such as edges and vertices~\cite{Kondratiev1967,Grisvard1985,Dauge1988}. Such effects are closely tied to various physical phenomena, including field enhancement near sharp tips, stress concentration in fracture mechanics, and Meixner-type edge conditions in electromagnetic scattering~\cite{Imbert2019,CostabelDauge2000}. In inverse problems, corners and edges also produce distinctive scattering signatures that underlie a number of uniqueness results for polyhedral and conic scatterers~\cite{LiuZou2006,BlastenPaivarintaSylvester2014,ElschnerHu2018}. Thus, understanding the local behavior of solutions near these nonsmooth regions is essential not only for the design of high-order numerical solvers of forward problems, but also for the accurate analysis in inverse problems.

Layer potential theory provides a powerful tool for solving elliptic boundary value problems and has been extensively studied over the past three decades~\cite{atkinson1997numerical,Kress2014,HsiaoWendland2021}. By representing solutions as single-layer or double-layer potentials on the boundary, the original problem is reformulated as boundary integral equations (BIEs) of the second kind, thereby reducing the dimensionality by one~\cite{LaiGreengardONeil2017,DongLaiLi2022}. Moreover, this approach naturally satisfies far-field radiation conditions in unbounded domains and is well-suited to fast algorithms such as the fast multipole method and related hierarchical solvers~\cite{GreengardRokhlin1987,SushnikovaGreengardONeilRachh2023}. These properties are particularly advantageous for three-dimensional problems.

For three-dimensional problems with smooth boundaries, classical Fredholm theory lays a solid foundation for well-posedness analysis~\cite{HsiaoWendland2021}. However, the theory becomes considerably more delicate for Lipschitz or piecewise smooth domains. On polygonal and polyhedral domains, solutions of elliptic partial differential equations admit asymptotic expansions that involve fractional powers and logarithmic terms~\cite{Kondratiev1967,Petersdorff1990,Costabel2002}. At the level of potential theory, the compactness of boundary integral operators (BIOs) is lost due to these geometric singularities. By recasting layer potentials on cones as Mellin convolution operators, \cite{QiaoNistor2012} establishes the invertibility of BIOs in suitable weighted Sobolev spaces. Subsequent studies have investigated the spectral properties of Neumann--Poincar\'e operators, characterizing the essential spectrum for rotationally symmetric cones~\cite{HelsingPerfekt2018} and locally dilation-invariant Lipschitz domains~\cite{ChandlerWildeHaggerPerfektVirtanen2023}. These results suggest that conical geometries offer a natural starting point for understanding singularities in three dimensions. Computationally, most existing high-order boundary integral equation methods for cornered domains have focused on two dimensions, including local mesh refinement combined with recursive compressed inverse preconditioning~\cite{helsing2008corner}, geometric regularization via corner rounding~\cite{epstein2016smoothed}, explicit singularity expansion methods~\cite{serkh2016solution,serkh2019detailed}, and kernel splitting on recursively refined panels~\cite{Yao2024}. 
In three dimensions, Li~\cite{Li2009} developed a high-order boundary integral solver for Laplace problems on surfaces with an isolated conical point, combining a partition-of-unity discretization with explicit treatment of the leading density singularities. Other high-order studies largely concern smooth surfaces or axisymmetric configurations~\cite{YingBirosZorin2006,HelsingKarlsson2014}. A systematic treatment that covers both edges and vertices, including their interaction on general piecewise smooth cones, is still lacking.

Motivated by this gap, we consider the electrostatic problem for the Laplace equation in three dimensions and investigate the singular behavior of the double-layer density for the Dirichlet problem on nonsmooth domains. Our focus is on locally dilation-invariant conical geometries, as defined in Section~\ref{subsec:Conical}. We first analyze two canonical models: the dihedral wedge and the axisymmetric circular cone, which provide analytically tractable geometries for edge and vertex singularities, respectively. We then turn to general piecewise smooth cones, where edge and vertex effects are coupled and no explicit formulas are available. For such geometries, we derive the associated nonlinear operator pencil and develop a Nystr\"om-based eigenvalue solver to compute the singular exponents. Our work therefore provides both analytic characterizations and a practical computational framework for singular exponents on three-dimensional conical geometries. These results are intended as a first step toward a broader singularity theory and a high-order numerical treatment of boundary integral equations on general three-dimensional nonsmooth surfaces.

The remainder of the paper is organized as follows. Section~\ref{sec:preliminaries} introduces the geometric settings and potential theory. In Section~\ref{sec:SE_Benchmarks}, we study the dihedral wedge and the circular cone, deriving the explicit singular exponents and local coefficient expansions. In Section~\ref{sec:SE_CaseIII}, we formulate the piecewise smooth cone model and develop a Nystr\"{o}m-based nonlinear eigenvalue solver. Section~\ref{sec:experiments} presents several numerical examples, including the trihedral cone, the regular $m$-hedral cones, and the non-polyhedral droplet-type cones. Finally, Section~\ref{sec:conclusions} concludes the paper and outlines several directions for future work. Some basic properties of Legendre functions and technical proofs are provided in the appendices.

\section{Mathematical preliminaries}\label{sec:preliminaries}
\subsection{Dilation-invariant conical domains}\label{subsec:Conical}
In this subsection, we establish the geometric settings for our analysis by introducing Lipschitz 3D corner domains defined in~\cite[Section~3.a]{Costabel2002}.

\begin{definition}[Lipschitz 3D corner domain]
Let $\Omega \subset \mathbb{R}^3$ be a bounded Lipschitz domain with boundary $S=\partial\Omega$. We say that $\Omega$ is a \emph{Lipschitz 3D corner domain} if, for every point $\bx\in S$, there exist a radius $r_{\bx}>0$ and a diffeomorphism $\chi_{\bx}$ sending $\bx$ to the origin such that $\chi_{\bx}\bigl(\Omega\cap B(\bx,r_{\bx})\bigr)$ coincides with a cone
\begin{equation}\label{eq:cone}
\Gamma_{\bx}=\{r\boldsymbol{\vartheta}:\ r>0,\ \boldsymbol{\vartheta}\in G_{\bx}\},
\end{equation}
where the spherical cross-section $G_{\bx}\subset\mathbb S^2$ is a Lipschitz domain on the unit sphere.
\end{definition}

Based on the topology of $G_{\bx}$, the boundary $S$ can be decomposed into three disjoint parts $S=\mathcal F\cup\mathcal E\cup\mathcal V$:
\begin{itemize}
\item Faces $\mathcal F$: For $\bx\in\mathcal F$, the tangent cone $\Gamma_{\bx}$ is a half-space, which equivalently means that $G_{\bx}$ is a hemisphere. The surface is locally smooth here.

\item Edges $\mathcal E$: For $\bx\in\mathcal E$, the tangent cone $\Gamma_{\bx}$ is a dihedral wedge of the form $\Gamma_{\bx}\cong C_\alpha\times\mathbb R$, where $C_\alpha\subset\mathbb R^2$ is a planar sector with opening angle $\alpha\in(0,2\pi)$. In this case, $G_{\bx}$ is a spherical lune.

\item Vertices $\mathcal V$: For $\bx\in\mathcal V$, the tangent cone $\Gamma_{\bx}$ is neither a half-space nor a dihedral wedge. Equivalently, $G_{\bx}$ forms a spherical polygon or, more generally, a spherical curvilinear polygon.
\end{itemize}

Rather than working directly on the global domain $\Omega$, we focus on the associated tangent cones, which capture the leading singular behavior near edges and vertices and are more amenable to analysis. Motivated by the classification above, our work first considers two models representing the prototypes of $\mathcal{E}$ and $\mathcal{V}$, respectively. The first is the dihedral wedge in Subsection~\ref{subsec:CaseI}, which captures translationally invariant edge singularities for $\bx \in \mathcal{E}$. The second is the axisymmetric circular cone in Subsection~\ref{subsec:CaseII}, which acts as a prototype for vertex singularities with rotational symmetry for $\bx \in \mathcal{V}$. As illustrated in Figure~\ref{fig:dihedral_and_cone}(a)(b), these two geometries share a key geometric property: dilation invariance.

\begin{figure}[htbp]
    \centering
    \begin{minipage}[t]{0.33\textwidth}
        \centering
        \includegraphics[width=\linewidth]{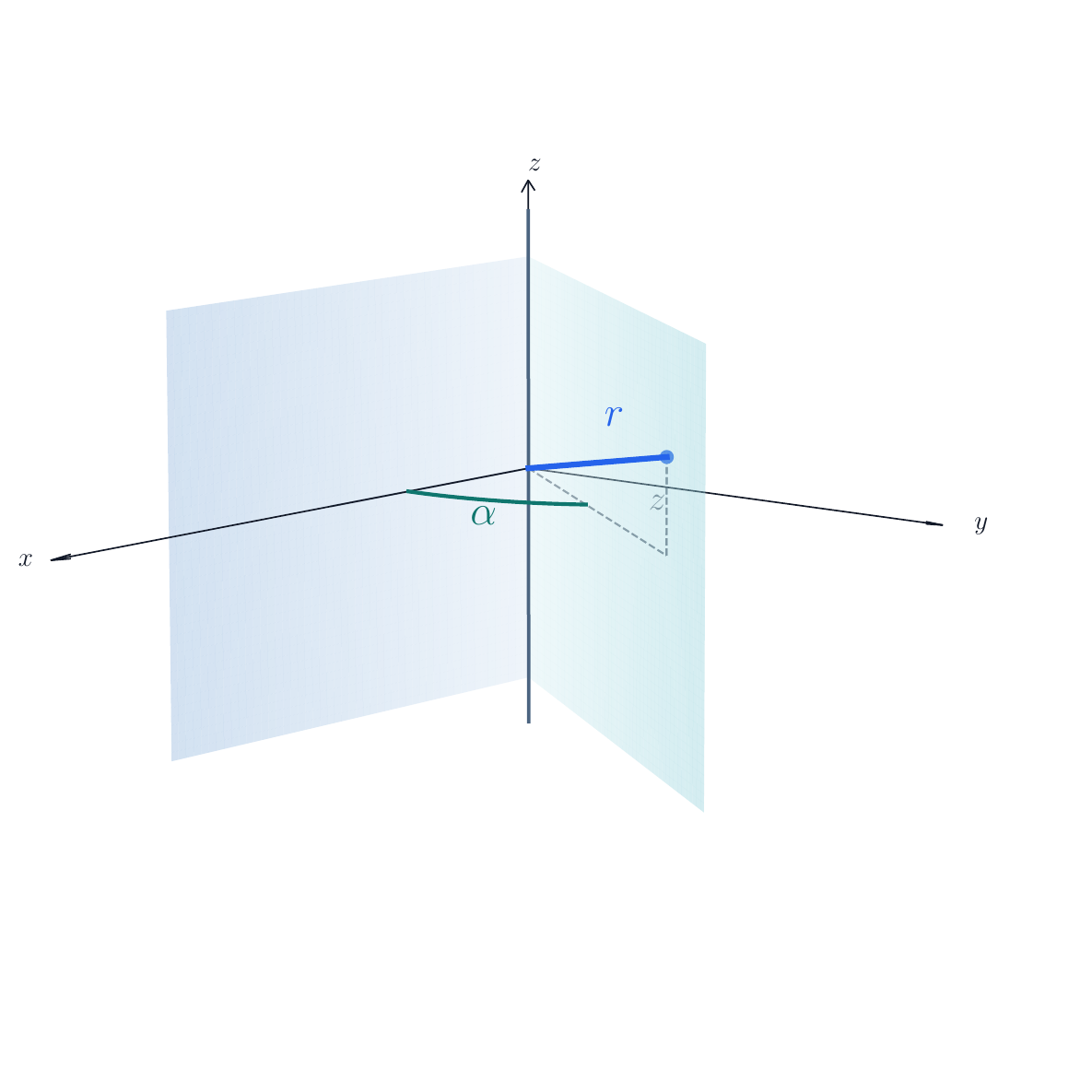}\\
        (a) Dihedral wedge
    \end{minipage}\hfill
    \begin{minipage}[t]{0.33\textwidth}
        \centering
        \includegraphics[width=\linewidth]{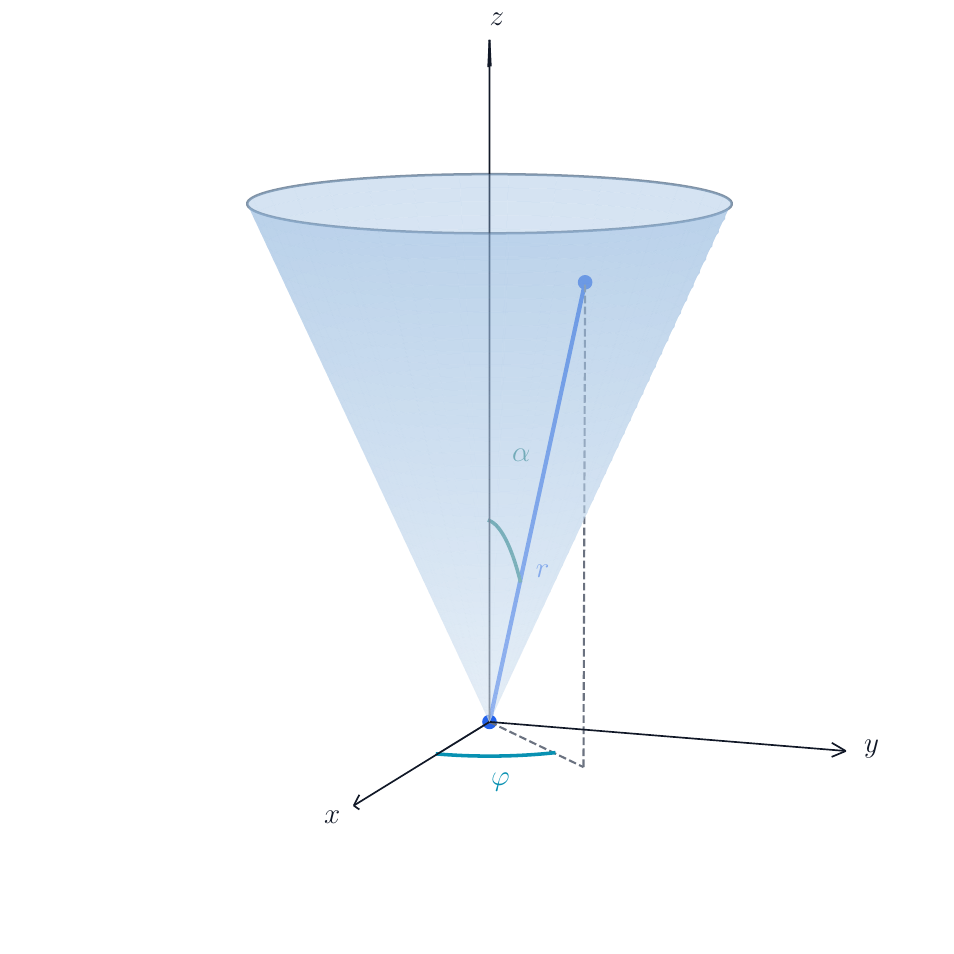}\\
        (b) Circular cone
    \end{minipage}
    \begin{minipage}[t]{0.33\textwidth}
        \centering
        \includegraphics[width=\linewidth]{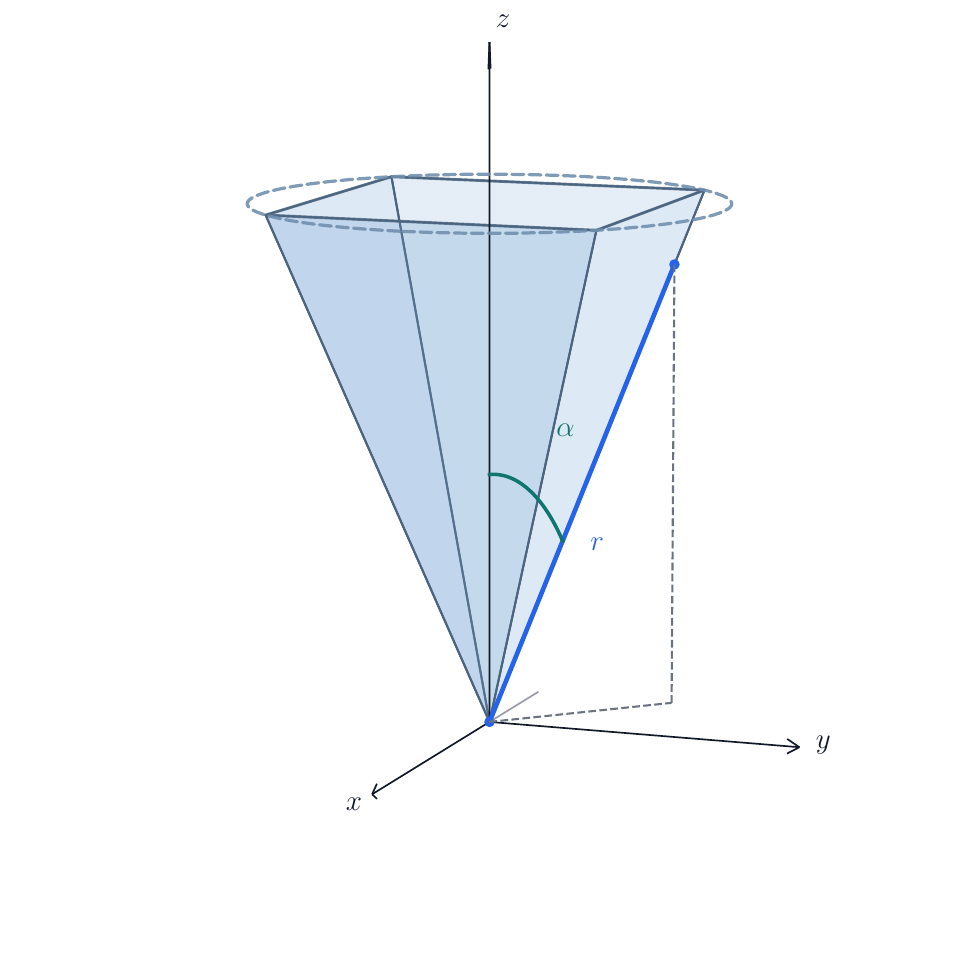}\\
        (c) Polyhedral cone
    \end{minipage}
	\caption{Representative dilation-invariant conical geometries: (a) an infinite dihedral wedge with opening angle $\alpha$; (b) an axisymmetric circular cone with half-apex angle $\alpha$; and (c) a regular polyhedral cone inscribed in a circular cone with half-apex angle $\alpha$.}
    \label{fig:dihedral_and_cone}
\end{figure}

\begin{definition}[Dilation-invariant]
A set $\Gamma\subset\mathbb R^3$ is dilation-invariant if it is invariant under the scaling transformation $\bx\mapsto \lambda \bx, \lambda>0$.
\end{definition}

This property preserves the homogeneity of the layer potential kernels under radial scaling and allows us to seek singular boundary densities in a separated-variable form, as detailed later in Eq.~\eqref{eq:density_separation}. Therefore, in this paper we focus on the general conical domains $\Gamma_{\bx}$ defined by Eq.~\eqref{eq:cone}, where the boundary of the spherical cross-section $\partial G_{\bx}$ consists of finitely many smooth arcs. As shown in Figure~\ref{fig:dihedral_and_cone}(c), the general piecewise smooth cone is the main conical model where we develop our systematic framework, unifying both the edge singularities and the vertex singularities. For further analysis, we impose the following assumptions on $\partial G_{\bx}$:
\begin{enumerate}[label=(\alph*), leftmargin=2.5em]
    \item Each boundary arc is of class $C^{\infty}$;
    \item Adjacent arcs intersect only at their endpoints;
    \item The corner points of $\partial G_{\bx}$ are non-degenerate, i.e., adjacent arcs meet with an angle in $(0,2\pi)$.
\end{enumerate}
Under these assumptions, the conical domain $\Gamma_{\bx}$ is piecewise smooth and dilation-invariant. This class is broad enough to include cones with curvilinear cross-sections, while still retaining sufficient geometric structure for the derivation of singular exponents and the construction of the nonlinear eigenvalue solver developed in Section~\ref{subsec:NEP_solver}.

\subsection{Boundary integral equations}
Consider the classical Dirichlet problem for the Laplace equation, which describes the electrostatic potential in a charge-free region:
\begin{equation}
\left\{
\begin{aligned}
    & \Delta u = 0  \quad \text{in } \Omega, \\
    & u = f  \quad \text{on } S,
\end{aligned}
\right.
\end{equation}
where $S=\partial\Omega$ denotes the boundary and $f$ is the prescribed smooth boundary data. Here $\Omega$ may be either a bounded interior domain or an unbounded exterior domain. In the exterior case, the solution is supplemented by the far-field decay condition~\cite{Colton2019}
\begin{equation}
    u(\bx) = o\left(1\right) \quad \text{as } |\bx| \to \infty.
\end{equation}

Let $\Phi(\bx,\by)$ be the fundamental solution to Laplace's equation in $\mathbb{R}^3$, given by
\begin{equation}
    \Phi(\bx,\by)=\frac{1}{4\pi|\bx-\by|}.
\end{equation}
If the surface $S$ carries a dipole distribution $\rho\in L^2(S)$, the associated double-layer potential (DLP) is defined by
\begin{equation}\label{eq:DLP}
u(\bx)=D[\rho](\bx)=\int_S \frac{\partial}{\partial n_{\by}}\Phi(\bx,\by)\,\rho(\by)\,ds_{\by},
\quad \bx\in\mathbb{R}^3 \setminus S,
\end{equation}
where $\bn_{\by}$ denotes the unit outward normal at $\by\in S$. Taking the boundary trace in the principal value sense yields the corresponding boundary integral operator
\begin{equation}\label{eq:BIO}
\mathcal{D}[\rho](\bx)=\operatorname{p.v.}\frac{1}{4\pi}\int_S
\frac{(\bx-\by)\cdot \bn_{\by}}{|\bx-\by|^3}\,\rho(\by)\,ds_{\by},
\quad \bx\in S.
\end{equation}

Classical jump relations imply that the nontangential limits of the double-layer potential satisfy
\[
\gamma^\sigma D[\rho] = \left(\frac{\sigma}{2}\mathcal{I} + \mathcal D\right)\rho, \quad \text{a.e. on } S,
\]
where $\sigma=\pm 1$ denotes the exterior and interior problems, respectively, $\gamma$ denotes the trace operator, and $\mathcal{I}$ is the identity operator. This leads to the second-kind boundary integral equations
\begin{equation}\label{eq:BIE}
\left(\frac{\sigma}{2}\mathcal{I}+\mathcal D\right)\rho = f, \quad \text{a.e. on } S.
\end{equation}

On smooth surfaces, the operator $\mathcal D$ is compact, and the corresponding density inherits the regularity of the boundary data. For piecewise smooth boundaries, however, even when the boundary data is smooth, the density $\rho$ generally develops edge and vertex singularities~\cite{Grisvard1985,Dauge1988,Costabel2002}. Our goal is to analyze these singularities on three-dimensional cornered domains.

\begin{remark}
For the exterior problem, the integral equation \eqref{eq:BIE} is solvable provided that the right-hand side $f$ is orthogonal to the kernel of the adjoint of $\frac{\sigma}{2}\mathcal{I}+\mathcal D$. Otherwise, a modified formulation is required~\cite{Colton2019}. In what follows, we assume that this orthogonality condition is satisfied, since the modified formulation can be analyzed in an analogous manner. Furthermore, the boundary value problem with Neumann data can be treated analogously using the adjoint properties of the double-layer potential and the normal derivative of the single-layer potential.
\end{remark}

\begin{remark}
The analysis of the circular cone and the general piecewise smooth cone relies heavily on the properties of Legendre functions. For the reader's convenience, the properties of Legendre functions used throughout this paper are collected in~\ref{app:legendre}.
\end{remark}

\section{Two canonical models}\label{sec:SE_Benchmarks}
In this section, we derive the leading singular behavior of the density function $\rho$ for the double-layer potential formulation of the Dirichlet problem on two canonical geometries: the dihedral wedge and the axisymmetric circular cone, shown in Figure~\ref{fig:dihedral_and_cone}(a) and (b), respectively.

\subsection{Infinite dihedral wedge}\label{subsec:CaseI}
As shown in Figure~\ref{fig:dihedral_and_cone}(a), let $S$ be the surface of an infinite dihedral wedge with an opening angle $\alpha \in(0,2 \pi)$, consisting of two half-planes intersecting along the $z$-axis:
\[
\begin{cases}
	S_1:= \; \{(r, 0, z): r \in [0, 1), z \in \mathbb{R}\},\\
	S_2:= \; \{(r \cos \alpha, r \sin \alpha, z): r \in [0, 1), z \in \mathbb{R}\},
\end{cases}
\]
where $\bn_1=(0,-1,0)$ and $\bn_2=(-\sin\alpha,\cos\alpha,0)$ are outward unit normal vectors on $S_1$ and $S_2$, respectively. Since the edge singularity is translationally invariant in $z$, it is natural to assume
\begin{equation}\label{eq:density_dihedral}
\rho(\by)=\begin{cases}
	r_y^\mu & \text{if } \by \in S_1,\\
	\eta\,r_y^\mu & \text{if } \by \in S_2,
\end{cases}
\end{equation}
for some $\mu>-1$.

\begin{proposition}\label{prop:DLP_map_dihedral}
    Given the density $\rho$ represented by~\eqref{eq:density_dihedral}, the double-layer integral operator~\eqref{eq:BIO} satisfies
    \begin{equation}\label{eq:DLP_diheral}
    \mathcal{D}(\rho)(\bx) =
    \begin{cases}
        -\dfrac{\eta\,\sin ((\pi-\alpha) \mu)}{2\sin (\pi \mu)} r_x^{\mu} - \displaystyle\sum_{k=1}^{\infty} \dfrac{\eta\,\sin (k \alpha)}{2\pi(\mu-k)} r_x^k, & \text{if } \bx \in S_1, \\[2ex]
        -\dfrac{\sin ((\pi-\alpha) \mu)}{2\sin (\pi \mu)} r_x^{\mu} - \displaystyle\sum_{k=1}^{\infty} \dfrac{\sin (k \alpha)}{2\pi(\mu-k)} r_x^k, & \text{if } \bx \in S_2.
    \end{cases}
    \end{equation}
\end{proposition}
\begin{proof}
    Let $\bx=(r_x,0,z_x)\in S_1$ and $\by=(r_y\cos\alpha,r_y\sin\alpha,z_y)\in S_2$. Then
    \begin{align}
        \mathcal{D}(\rho)(\bx) &= \frac{1}{4\pi} \int_0^1 \eta\,r_y^\mu \left( \int_{-\infty}^{\infty} \frac{-r_x \sin \alpha}{(r_x^2 + r_y^2 - 2r_x r_y \cos \alpha + (z_x-z_y)^2)^{3/2}} dz_y \right) dr_y \notag\\
        &= -\frac{\eta\,r_x \sin \alpha}{2\pi} \int_0^1 \frac{r_y^\mu}{r_x^2 - 2 r_x r_y \cos \alpha + r_y^2} dr_y \label{eq:dihedral_1}\\
        &= -\frac{\eta\,r_x \sin \alpha}{2\pi} \left(\frac{\pi\sin ((\pi-\alpha) \mu)}{\sin (\pi \mu) \sin \alpha} r_x^{\mu-1} + \sum_{k=1}^{\infty} \frac{\sin (k \alpha)}{(\mu-k) \sin \alpha} r_x^{k-1}\right) \label{eq:dihedral_2}\\
        & = -\frac{\eta\,\sin ((\pi-\alpha) \mu)}{2\sin (\pi \mu)} r_x^{\mu} - \sum_{k=1}^{\infty} \frac{\eta\,\sin (k \alpha)}{2\pi(\mu-k)} r_x^k, \label{eq:dihedral_3}
    \end{align}
    where Eq.~\eqref{eq:dihedral_1} follows from the integral
    \[
    \int_{-\infty}^{\infty} \frac{1}{(A^2 + z^2)^{3/2}} dz = \frac{2}{A^2} \mbox{ for }A\ne 0,
    \]
    and Eq.~\eqref{eq:dihedral_2} is obtained via the Mellin transform in~\cite[Lemma 3.2]{xie2025singularity}. The case for $\bx \in S_2$ can be analyzed analogously.
\end{proof}

These results are fully consistent with the classical two-dimensional analysis of Serkh and Rokhlin~\cite{serkh2016solution}, whose detailed continuation~\cite{serkh2019detailed} removes the restriction to all but finitely many opening angles by admitting non-integer powers multiplied by logarithms. Reformulating the even and odd cases yields the following theorem, and the detailed proof can be found in~\ref{app:dihedral_wedge}.

\begin{theorem}\label{thm:Dihedral_wedge}
    For the Dirichlet problem on an infinite dihedral wedge with opening angle $\alpha \in (0, 2\pi)$, fix $\alpha$ outside some finite exceptional set, and assume that the even and odd parts of the induced cross-sectional boundary data are representable by convergent Taylor series. Then, for every positive integer $N$, there exist unique coefficients $b_0,\dots,b_N$ and $c_0,\dots,c_N$ such that the truncated density of the double-layer potential
    \begin{equation}\label{eq:density_dihedral_truncated}
    \rho^{(N)}(\by) =
    \begin{cases}
        \displaystyle\sum_{k=0}^{N} b_k r_y^{k\pi/\alpha} + \displaystyle\sum_{k=0}^{N} c_k r_y^{k\pi/(2\pi-\alpha)}, & \quad \by\in S_1, \\[2ex]
        \displaystyle\sum_{k=0}^{N} (-1)^{k+1} b_k r_y^{k\pi/\alpha} + \displaystyle\sum_{k=0}^{N} (-1)^k c_k r_y^{k\pi/(2\pi-\alpha)}, & \quad \by\in S_2,
    \end{cases}
    \end{equation}
    satisfies the boundary integral equation~\eqref{eq:BIE} up to a residual of order $O(r^{N+1})$ near the edge.
\end{theorem}

The above analysis also indicates that the local density near the edges of the dihedral wedge is generated by two fractional power families
\begin{equation}\label{eq:D_branches}
    \mathcal{S}_{\text{dihedral}} = \left\{\frac{k\pi}{\alpha} : k \in \mathbb{N}_0 \right\} \bigcup \left\{ \frac{k\pi}{2\pi-\alpha} : k \in \mathbb{N}_0 \right\}.
\end{equation}
In particular, the leading singular exponent is given by
\begin{equation}\label{eq:D_leading}
    \mu_{\text{dihedral}}(\alpha) = \min \left( \frac{\pi}{\alpha},\, \frac{\pi}{2\pi-\alpha} \right).
\end{equation}

\subsection{Axisymmetric circular cone}\label{subsec:CaseII}
As shown in Figure~\ref{fig:dihedral_and_cone}(b), let $S$ be the surface of a circular cone parameterized by
\[
(r \sin \alpha \cos \varphi, r \sin \alpha \sin \varphi, r \cos \alpha),\quad r \in[0,1), \varphi \in[0,2 \pi),
\]
where $\alpha \in (0,\pi)$ is the half-apex angle of the cone. Since the geometry is rotationally invariant, we assume
\begin{equation}\label{eq:density_cone}
    \rho(\by)=r_y^\mu e^{i m \varphi_y}
\end{equation}
for some $\mu>-1$ and $m \in \mathbb{Z}$, where the power function $r^\mu$ captures the radial singularity and $e^{im\varphi}$ represents the Fourier mode.

\subsubsection{Double-layer integral operator}
For the source point $\by = (r_y \sin \alpha \cos \varphi_y,\; r_y \sin \alpha \sin \varphi_y,\; r_y \cos \alpha)\in S$, the unit outward normal vector and the surface element are $\bn_y = (\cos \alpha \cos \varphi_y,\; \cos \alpha \sin \varphi_y,\; -\sin \alpha)$, $dS_{\by} = r_y \sin \alpha \, dr_y d\varphi_y$. Then at the target point $\bx = (r_x \sin \alpha \cos \varphi_x,\; r_x \sin \alpha \sin \varphi_x,\; r_x \cos \alpha) \in S$, the DLP operator~\eqref{eq:BIO} becomes
\begin{equation}\label{eq:BIO_cone}
	\mathcal{D}(\rho)(\bx)=-\frac{\sin^2 \alpha \cos \alpha}{4\pi}\int_0^{2 \pi} \int_0^1 \frac{r_x r_y [1 - \cos(\varphi_x - \varphi_y)] \rho(\by)}{\left(r_x^2-2 r_x r_y \chi(x, y, \alpha)+r_y^2\right)^{3 / 2}} d r_y d \varphi_y,
\end{equation}
where $\chi(x, y, \alpha)$ is the cosine of the angle between $\bx$ and $\by$ given by
\[
\chi(x, y, \alpha) = \sin^2 \alpha \cos(\varphi_x - \varphi_y) + \cos^2 \alpha \in [\cos2\alpha,1] \subset [-1,1].
\]

Observing that $\frac{\partial \chi}{\partial \alpha} = \sin(2 \alpha)[\cos(\varphi_x-\varphi_y)-1]$, we can rewrite the operator~\eqref{eq:BIO_cone} as
\begin{equation}\label{eq:BIO_cone_dalpha}
	\mathcal{D}(\rho)(\bx) = \frac{\sin \alpha}{8\pi} \frac{\partial}{\partial \alpha} \int_0^{2 \pi} \int_0^1 \frac{\rho(\by)}{\left(r_x^2-2 r_x r_y \chi(x, y, \alpha)+r_y^2\right)^{1 / 2}} d r_y d \varphi_y.
\end{equation}
Before introducing the main result, we first present an integral~\cite{Bateman1954} that will be frequently used.

\begin{lemma}\label{lem:Mellin_transform}
    For $\Re(\mu) \in (-1,0)$, the Mellin transform yields
    \begin{equation}\label{eq:Mellin_transform}
        \int_0^{\infty} \frac{t^{\mu}}{(1-2x t + t^2)^{1/2}} dt = -\pi \csc(\pi \mu) P_{\mu}(-x), \quad -1<x<1,
    \end{equation}
    where $P_{\mu}(x)$ denotes the Legendre function of the first kind.
\end{lemma}

\begin{remark}
For $x\in(-1,1)$, the integral in~\eqref{eq:Mellin_transform}
converges only in the strip $\Re(\mu)\in(-1,0)$, and the right-hand side defines its meromorphic continuation in the parameter $\mu$, with possible simple poles at $\mu\in\mathbb Z$. As $x\to -1$, the left-hand side remains integrable and $P_\mu(1)=1$, so this identity extends naturally to this endpoint. However, as $x\to 1$, the integrand behaves like $|t-1|^{-1}$ near $t=1$, and $P_\mu(-x)$ develops a logarithmic singularity. Since the resulting logarithmic singularity is integrable in the angular variable, this endpoint can be handled by a limiting interpretation.
\end{remark}

\begin{proposition}\label{prop:DLP_map_cone}
    Suppose the density $\rho$ can be represented by~\eqref{eq:density_cone}, then the double-layer integral operator~\eqref{eq:BIO} satisfies
    \begin{equation}\label{eq:DLP_cone}
    \mathcal{D}(\rho)(\bx)
    =
    \frac{\pi \sin^2 \alpha}{4\sin(\pi\mu)} A_{\mu,m}(\alpha)\, r_x^{\mu}e^{im\varphi_x}
    +
    \frac{\sin^2 \alpha}{4}\sum_{n=|m|}^{\infty}
    \frac{B_{n,m}(\alpha)}{n-\mu}\, r_x^n e^{im\varphi_x},
    \end{equation}
    where
    \begin{align*}
    A_{\mu,m}(\alpha)
    &= P_{\mu}^{-|m|}(-\cos\alpha)\,(P_{\mu}^{|m|})'(\cos\alpha)
    - (P_{\mu}^{-|m|})'(-\cos\alpha)\,P_{\mu}^{|m|}(\cos\alpha),\\
    B_{n,m}(\alpha)
    &= 2\frac{(n-|m|)!}{(n+|m|)!}P_n^{|m|}(\cos\alpha)(P_n^{|m|})'(\cos\alpha).
    \end{align*}
\end{proposition}
\begin{proof}
    Consider the radial integral
    \begin{align}
        I_1 &\triangleq \int_0^{1} \frac{r_y^{\mu}}{\left(r_x^2-2 r_x r_y \chi(x, y, \alpha)+r_y^2\right)^{1/2}} d r_y = r_x^{\mu} \int_0^{1/r_x} \frac{t^{\mu}}{\left(1-2 \chi t+t^2\right)^{1/2}} dt \notag\\
        & = r_x^{\mu}\left(\int_0^{\infty} \frac{t^{\mu}}{\left(1-2 \chi t+t^2\right)^{1/2}} dt - \int_{1/r_x}^{\infty} \frac{t^{\mu}}{\left(1-2 \chi t+t^2\right)^{1/2}} dt\right). \label{eq:integral1}
    \end{align}
    By the generating function of Legendre functions~\cite[P361, Eq.(14.7.19)]{NIST2010}, we obtain
    \begin{align}
        \int_{1/r_x}^{\infty} \frac{t^{\mu}}{(1-2 \chi t + t^2)^{1/2}} dt &= \int_{0}^{r_x} \frac{t^{-\mu-1}}{(1-2 \chi t + t^2)^{1/2}} dt = \int_{0}^{r_x} t^{-\mu-1} \sum_{n=0}^{\infty} P_n(\chi) t^n dt \notag\\
        &= \sum_{n=0}^{\infty} P_n(\chi) \int_{0}^{r_x} t^{n-\mu-1} dt
        = \sum_{n=0}^{\infty} \frac{P_n(\chi)}{n-\mu} r_x^{n-\mu}, \quad |\chi| \leq 1. \label{eq:integral12}
    \end{align}
    Combined with~\eqref{eq:Mellin_transform}, we have
    \begin{equation}
        I_1 = -\pi \csc(\pi \mu) P_{\mu}(-\chi)r_x^{\mu} - \sum_{n=0}^{\infty} \frac{P_n(\chi)}{n-\mu} r_x^n . \label{eq:integral13}
    \end{equation}

    To integrate over $\varphi_y$, we apply the addition theorem~\cite[P370, Eq.(14.18.1), Eq.(14.18.2)]{NIST2010} with $\theta_1=\pi-\alpha$, $\theta_2=\alpha$ and $\phi=\varphi_y - \varphi_x+\pi$, which gives
    \begin{equation}\label{eq:addition1}
    P_{\mu}(-\chi)
    =P_{\mu}(-\cos\alpha)P_{\mu}(\cos\alpha)
    +2\sum_{k=1}^{\infty} P_{\mu}^{-k}(-\cos\alpha)P_{\mu}^{k}(\cos\alpha)\cos(k(\varphi_y - \varphi_x)).
    \end{equation}
    With $\theta_1=\theta_2=\alpha$ and $\phi=\varphi_y - \varphi_x$, it holds
    \begin{equation}\label{eq:addition2}
    P_{n}(\chi) = (P_n(\cos\alpha))^2 + 2\sum_{k=1}^{n} (-1)^k P_n^{-k}\left(\cos \alpha\right) P_n^k\left(\cos \alpha\right) \cos (k (\varphi_y - \varphi_x)).
    \end{equation}
    Substituting~\eqref{eq:addition1} and~\eqref{eq:addition2} into~\eqref{eq:integral13} and using the orthogonality of Fourier modes, we have
    \begin{align}
        I_2 \triangleq &\, \int_0^{2 \pi} \left(-\pi \csc(\pi \mu) P_{\mu}(-\chi) r_x^{\mu} - \sum_{n=0}^{\infty} \frac{P_n(\chi)}{n-\mu} r_x^n\right) e^{i m \varphi_y} d \varphi_y \notag\\
        =&\; \left(-2\pi^2 \csc(\pi \mu) P_{\mu}^{-|m|}(-\cos\alpha)P_{\mu}^{|m|}(\cos\alpha)r_x^{\mu}   - 2\pi\sum_{n=|m|}^{\infty}\frac{(-1)^{|m|} P_n^{-|m|}\left(\cos \alpha\right) P_n^{|m|}\left(\cos \alpha\right)}{n-\mu} r_x^n \right) e^{im\varphi_x}.
        \label{eq:integral2}
    \end{align}
    Substituting~\eqref{eq:integral2} into~\eqref{eq:BIO_cone_dalpha} yields the final result~\eqref{eq:DLP_cone}.
\end{proof}

\begin{remark}
    The above formula holds when $\Re(\mu)>-1$ and $\mu\notin\mathbb Z$. When $\mu=m\in\mathbb Z$, \eqref{eq:integral13} can be understood as the limiting value
    \begin{equation}
        \lim_{\mu \to m} I_1 = - P_m(\chi) r_x^m \ln(r_x) - (-1)^m \left. \frac{\partial P_\mu(-\chi)}{\partial \mu} \right|_{\mu=m} r_x^m  - \sum_{\substack{n=0 \\ n \neq m}}^{\infty} \frac{P_n(\chi)}{n-m} r_x^n. \label{eq:integral13_limit}
    \end{equation}
\end{remark}

\subsubsection{Leading singular behavior}
For the Dirichlet problem with smooth boundary data, the leading-order term in $r_x^\mu$ must vanish, yielding the following transcendental equation for the singularity exponent $\mu$ and the mode $m$:
\begin{equation}\label{eq:transcendental}
	V_{{\mu,m}}(\alpha) = \frac{\sigma}{2} + \frac{\pi \sin^2 \alpha}{4\sin(\pi\mu)} A_{\mu,m}(\alpha) = 0,
\end{equation}
whose roots determine the leading singular behavior of the density near the vertex of the circular cone.

\begin{proposition}\label{prop:cone_branches}
    For $\mu\notin\mathbb Z$, we obtain
    \begin{equation}\label{eq:V_reduce}
        V_{{\mu,m}}(\alpha) =
        \begin{cases}
            -\dfrac{\pi\sin^2\alpha}{2\sin(\pi\mu)}\,\kappa_{\mu,m}(P_{\mu}^{-|m|})'(-\cos\alpha)\,P_{\mu}^{-|m|}(\cos\alpha), & \text{for the interior problem}, \\[2ex]
            \dfrac{\pi\sin^2\alpha}{2\sin(\pi\mu)}\,\kappa_{\mu,m}P_{\mu}^{-|m|}(-\cos\alpha)\,(P_{\mu}^{-|m|})'(\cos\alpha), & \text{for the exterior problem},
        \end{cases}
    \end{equation}
    where
    \[
    \kappa_{\mu,m}:=(-1)^{|m|}\frac{\Gamma(\mu+|m|+1)}{\Gamma(\mu-|m|+1)}.
    \]
\end{proposition}
\begin{proof}
    Throughout the proof we write $x=\cos\alpha$, so that $1-x^2=\sin^2\alpha$ for $\alpha\in(0,\pi)$. By the connection formula~\cite[P362, Eq.(14.9.3)]{NIST2010}, we can simplify $A_{\mu,m}(\alpha)$ as
    \[
    A_{\mu,m}(\alpha)
    =
    \kappa_{\mu,m}
    \Big(
    P_{\mu}^{-|m|}(-x)\,(P_{\mu}^{-|m|})'(x)
    -
    (P_{\mu}^{-|m|})'(-x)\,P_{\mu}^{-|m|}(x)
    \Big).
    \]
    Using the Wronskian identity~\cite[P352, Eq.(14.2.3)]{NIST2010}
    \[
    \mathscr{W}\{P_{\mu}^{-|m|}(x),P_{\mu}^{-|m|}(-x)\}
    =
    \frac{2}{\Gamma(|m|-\mu)\Gamma(\mu+|m|+1)(1-x^2)}
    \]
    and applying Euler's reflection formula
    \[
    \Gamma(|m|-\mu)\Gamma(\mu-|m|+1)=(-1)^{|m|+1}\frac{\pi}{\sin(\pi\mu)},
    \]
    we can finally obtain~\eqref{eq:V_reduce}.
\end{proof}

Since $0<\alpha<\pi$ and $\mu\notin\mathbb Z$, \eqref{eq:V_reduce} yields two branches of singularity exponents on a circular cone
\begin{equation}\label{eq:cone_roots}
    \mathcal S_m(\alpha):=
    \begin{cases}
    \mathcal{S}^{(D)}_{m}(\alpha)\cup \mathcal{S}^{(N)}_{m}(\pi-\alpha), & \text{for the interior problem},\\
    \mathcal{S}^{(D)}_{m}(\pi-\alpha)\cup \mathcal{S}^{(N)}_{m}(\alpha), & \text{for the exterior problem},
    \end{cases}
\end{equation}
where
\begin{align*}
\mathcal{S}^{(D)}_{m}(\alpha)
&:=
\left\{\mu\in\mathbb{R}_+\setminus\mathbb{Z}:\,
P_\mu^{m}(\cos\alpha)=0\right\},\\
\mathcal{S}^{(N)}_{m}(\alpha)
&:=
\left\{\mu\in\mathbb{R}_+\setminus\mathbb{Z}:\,
(P_\mu^{m})'(\cos\alpha)=0\right\}.
\end{align*}
In particular, the leading singular exponent for the $m$-th mode is given by
\begin{equation}\label{eq:V_leading_m}
\mu_{\text{circular},m}(\alpha) = \min \mathcal{S}_{m}(\alpha),
\end{equation}
and the overall leading singularity is given by
\begin{equation}\label{eq:V_leading}
\mu_{\text{circular}}(\alpha) = \min_{m\in\mathbb N_0}\ \mu_{\text{circular},m}(\alpha).
\end{equation}

\subsubsection{An explicit coefficient mapping}
Let the boundary data near the vertex be given by a Fourier series. When it is generated by a smooth function, each Fourier mode satisfies the following property.

\begin{lemma}\label{lem:cone_bdry_data}
    Let $g(r,\varphi)=\sum_{m\in\mathbb Z} g_m(r)e^{im\varphi}$ be the trace on the circular cone of a function $G\in C^\infty$ defined in a neighborhood of the vertex. Then, for each $m\in\mathbb Z$ and each positive integer $N$, there exist coefficients $d_{m,0},\dots,d_{m,N-1}$ such that
    \begin{equation}\label{eq:bdry_taylor_expansion}
        g_m(r)=\sum_{j=0}^{N-1} d_{m,j} r^{|m|+j}+O(r^{|m|+N}),
        \quad r\to0.
    \end{equation}
    In particular, $g_m(r)=O(r^{|m|})$ as $r\to0$.
\end{lemma}
\begin{proof}
     Expanding $G$ in Cartesian monomials and rewriting each term as $(x+iy)^p(x-iy)^q z^c$ where $p,q,c\in\mathbb{N}_0$, its restriction to the cone takes the form
    \[
    C_{p,q,c}(\alpha)\, r^{p+q+c} e^{i(p-q)\varphi}.
    \]
    Hence, a contribution to the $m$-th Fourier mode must satisfy $p-q=m$, which implies $p+q+c\ge |p-q|=|m|$. Therefore, the first $N$ powers in the $m$-th Fourier mode occur at orders $r^{|m|},\dots,r^{|m|+N-1}$, and the remainder is $O(r^{|m|+N})$.
\end{proof}

To construct the coefficient map, we fix the problem type $\sigma$ and the Fourier mode $m$. Let $\mu_{m,1}<\cdots<\mu_{m,N}$ be the first $N$ noninteger elements of $\mathcal S_m(\alpha)$, and consider the truncated singular expansion
\begin{equation}\label{eq:density_singular_expansion}
    \rho_m^{(N)}(r,\varphi):=\sum_{k=1}^{N} a_{m,k} r^{\mu_{m,k}}e^{im\varphi}.
\end{equation}
By \eqref{eq:DLP_cone}, each basis function $r^{\mu_{m,k}}e^{im\varphi}$ generates only powers $r^n e^{im\varphi}$ with $n\ge |m|$:
\begin{equation}\label{eq:cone_matrix_map}
\left(\frac{\sigma}{2} \mathcal I + \mathcal D\right)\rho_m^{(N)}(\bx)
=
\sum_{n=|m|}^{\infty}
\left(
\sum_{k=1}^{N}
\frac{\sin^2\alpha}{4}\,
\frac{B_{n,m}(\alpha)}{n-\mu_{m,k}}\,
a_{m,k}
\right)
r_x^n e^{im\varphi_x}.
\end{equation}

Fix $m$ and $N$, and let $n_j:=|m|+j-1$. Define the mapping matrix $M^{(m,N)}(\alpha) \in \mathbb{R}^{N \times N}$ by
\[
M_{j,k}^{(m,N)}(\alpha)
:=
\frac{\sin^2\alpha}{4}\,
\frac{B_{n_j,m}(\alpha)}{n_j-\mu_{m,k}},
\quad 1\le j, k\le N.
\]
Matching the coefficients of $r^{|m|},\dots,r^{|m|+N-1}$ in~\eqref{eq:cone_matrix_map} with those in~\eqref{eq:bdry_taylor_expansion} yields
\[
M^{(m,N)}(\alpha)\,\mathbf a_m^{(N)}=\mathbf d_m^{(N)},
\]
where $\mathbf a_m^{(N)}:=(a_{m,1},\dots,a_{m,N})^\top$, $\mathbf d_m^{(N)}:=(d_{m,0},\dots,d_{m,N-1})^\top$.

\begin{lemma}\label{lem:cone_matrix_invertible}
    For fixed $m\neq 0$ and $N$, define $\beta_j(\alpha):=\dfrac{\sin^2\alpha}{4}B_{n_j,m}(\alpha)$, $j=1, \dots, N$. Then the set
    \begin{equation}\label{eq:exceptional_set_cone}
    E_{m,N}:=\bigcup_{j=1}^N
    \left\{
    \alpha\in(0,\pi): \beta_j(\alpha)=0
    \right\}
    \end{equation}
    is finite. Moreover, for every $\alpha\in(0,\pi)\setminus E_{m,N}$, $M^{(m,N)}(\alpha)$ is invertible and
    \[
    \det M^{(m,N)}(\alpha)
    =
    \Big(\prod_{j=1}^N \beta_j(\alpha)\Big)
    \frac{\prod_{1\le j<\ell\le N}(n_\ell-n_j)(\mu_{m,j}-\mu_{m,\ell})}
    {\prod_{j,k=1}^N (n_j-\mu_{m,k})}.
    \]
\end{lemma}
\begin{proof}
    Since both $P_{n_j}^{|m|}(x)$ and $(P_{n_j}^{|m|})'(x)$ have only finitely many zeros in $x\in(-1,1)$, the finiteness of set $E_{m,N}$ is clear. Moreover, for $\alpha\notin E_{m,N}$, we can easily prove that $n_j$ and $\mu_{m,k}$ are distinct for every $j,k=1,\dots,N$. Hence,
    \[
    M^{(m,N)}(\alpha)=\operatorname{diag}(\beta_1(\alpha),\dots,\beta_N(\alpha))\,C,
    \]
    where $C_{j,k}=1/(n_j-\mu_{m,k})$ is a Cauchy matrix. The Cauchy determinant formula gives
    \[
    \det C=
    \frac{\prod_{1\le j<\ell\le N}(n_\ell-n_j)(\mu_{m,j}-\mu_{m,\ell})}
    {\prod_{j,k=1}^N (n_j-\mu_{m,k})}\neq 0.
    \]
    Multiplying by $\prod_{j=1}^N\beta_j(\alpha)$ yields the claimed results.
\end{proof}

However, since $B_{0,0}(\alpha)=0$, the mode $m=0$ requires a minor modification. Here we add a constant basis to the density ansatz:
\begin{equation}\label{eq:density_singular_expansion_0}
\rho_0^{(N)}(r)
=
a_{0,0}+\sum_{k=1}^N a_{0,k}r^{\mu_{0,k}}.
\end{equation}

\begin{lemma}\label{lem:constant_basis_mapping}
For the mode $m=0$, the constant density $\rho_0(r,\varphi)= 1$ yields
\begin{equation}\label{eq:D_constant_mapping}
\mathcal D(\rho_0)(r_x,\varphi_x) = -\frac{\cos\alpha}{2} + \frac{\sin^2\alpha}{4} \sum_{n=1}^{\infty}
\frac{B_{n,0}(\alpha)}{n}r_x^n.
\end{equation}
\end{lemma}
\begin{proof}
Let $\mu\to0$ in~\eqref{eq:integral13_limit}. Since
$\left.\dfrac{\partial}{\partial \mu}P_\mu(x)\right|_{\mu=0} = \ln\dfrac{1+x}{2}$, we obtain
\[
I_1 = -\ln r_x -\ln\frac{1-\chi}{2} - \sum_{n=1}^{\infty}\frac{P_n(\chi)}{n}r_x^n .
\]
Substituting into \eqref{eq:BIO_cone_dalpha} and using addition theorem~\eqref{eq:addition2}, we have
\begin{align}
	\mathcal{D}(\rho_0)(\bx) &= \frac{\sin \alpha}{8\pi} \frac{\partial}{\partial \alpha} \int_0^{2 \pi} \left(-\ln\frac{1-\chi}{2} - \sum_{n=1}^{\infty}\frac{P_n(\chi)}{n}r_x^n\right) d \varphi_y \notag\\
    &= \frac{\sin \alpha}{8\pi} \frac{\partial}{\partial \alpha} \left(-4\pi\ln(\sin\alpha)-2\pi \sum_{n=1}^{\infty} \frac{P_n^2(\cos\alpha)}{n}r_x^n\right) \notag\\
    &= -\frac{\cos\alpha}{2} + \frac{\sin^2\alpha}{4} \sum_{n=1}^{\infty} \frac{B_{n,0}(\alpha)}{n}\,r_x^n.
\end{align}
\end{proof}

Thus, matching the coefficients of $r^0,r^1,\dots,r^N$ leads to the system
\[
\widehat M^{(0,N)}(\alpha)\widehat{\mathbf a}_0^{(N)}
=
\widehat{\mathbf d}_0^{(N)},
\]
where
\[
\widehat M^{(0,N)}(\alpha) =
\begin{pmatrix}
\frac{\sigma-\cos\alpha}{2} & 0 & \cdots & 0 \\
\frac{\sin^2\alpha}{4}\frac{B_{1,0}}{1} & \frac{\sin^2\alpha}{4}\frac{B_{1,0}}{1-\mu_{0,1}} & \cdots & \frac{\sin^2\alpha}{4}\frac{B_{1,0}}{1-\mu_{0,N}} \\
\vdots & \vdots & & \vdots \\
\frac{\sin^2\alpha}{4}\frac{B_{N,0}}{N} & \frac{\sin^2\alpha}{4}\frac{B_{N,0}}{N-\mu_{0,1}} & \cdots & \frac{\sin^2\alpha}{4}\frac{B_{N,0}}{N-\mu_{0,N}}
\end{pmatrix}
\]
and
\[
\widehat{\mathbf a}_0^{(N)}
= (a_{0,0},a_{0,1},\dots,a_{0,N})^T,
\quad
\widehat{\mathbf d}_0^{(N)}
= (d_{0,0},d_{0,1},\dots,d_{0,N})^T.
\]
The following result is analogous to Lemma~\ref{lem:cone_matrix_invertible}, with the extra first column accounted for by the factor $(\sigma-\cos\alpha)/2$.

\begin{lemma}\label{lem:cone_matrix_invertible_0}
    For fixed $N$, define $\widehat\beta_j(\alpha):=\dfrac{\sin^2\alpha}{4}B_{j,0}(\alpha)$, $j=1, \dots, N$. Then the set
    \begin{equation}\label{eq:exceptional_set_cone_0}
    E_{0,N} := \bigcup_{j=1}^N \{\alpha\in(0,\pi):\widehat\beta_j(\alpha)=0\},
    \end{equation}
    is finite. Moreover, for all $\alpha\in(0,\pi)\setminus E_{0,N}$, the matrix $\widehat M^{(0,N)}(\alpha)$ is invertible and
    \[
    \det \widehat M^{(0,N)}(\alpha)
    =
    \frac{\sigma-\cos\alpha}{2}
    \Big(\prod_{j=1}^N \widehat\beta_j(\alpha)\Big)
    \frac{\prod_{1\le j<\ell\le N}(\ell-j)(\mu_{0,j}-\mu_{0,\ell})}
    {\prod_{j,k=1}^N(j-\mu_{0,k})}.
    \]
\end{lemma}

Consequently, for each Fourier mode $m$, the singular coefficients of the density can be uniquely determined by the Taylor coefficients of the boundary data whenever $\alpha\notin E_{m,N}$, which is summarized in the following theorem.

\begin{theorem}\label{thm:Circular_cone}
    Fix positive integers $M$ and $N$. For the Dirichlet problem on a circular cone with half-apex angle $\alpha\in(0,\pi)\setminus E_{M,N}$, where $E_{M,N} = \bigcup_{|m|\le M} E_{m,N}$, assume that the boundary data is smooth near the vertex. Suppose that $\mu_{m,1}<\cdots<\mu_{m,N}$ are the first $N$ noninteger elements of $\mathcal S_m(\alpha)$. Then there exist unique coefficients $a_{m,j}$ such that the truncated density
    \[
    \rho^{(N)}(r,\varphi)=a_{0,0}+\sum_{m=-M}^M\sum_{k=1}^{N} a_{m,k} r^{\mu_{m,k}}e^{im\varphi}
    \]
    satisfies the boundary integral equation up to a residual of order $O(r^{N+1})$ near the vertex:
    \[
    \left(\frac{\sigma}{2}\mathcal I+\mathcal D\right)\rho^{(N)}(r,\varphi)
    = \sum_{m=-M}^M g_m(r)e^{im\varphi}+R^{(M,N)}(r,\varphi),
    \]
    where $\sigma=\pm 1$ denotes the exterior or interior problems and
    \[
    R^{(M,N)}(r,\varphi)
    =
    O(r^{N+1})
    +
    \sum_{\substack{|m|\le M, m\ne0}}
    O(r^{|m|+N})e^{im\varphi},
    \quad r\to0.
    \]
\end{theorem}

\section{Singularities of general piecewise smooth cones}\label{sec:SE_CaseIII}
We now consider general piecewise smooth cones, as shown in Figure~\ref{fig:dihedral_and_cone}(c). This model serves as a generic prototype for exploring the interaction between edge and vertex singularities. Specifically, it represents conical domains whose spherical cross-section boundaries are curvilinear polygons on the unit sphere.

\subsection{General piecewise smooth cones}\label{subsec:Case_III}
Let $S$ be the surface of a general piecewise smooth cone with its vertex at the origin. We assume that its spherical cross-section boundary $\partial G\subset \mathbb S^2$ admits a piecewise arc-length parameterization
\[
\bt(s)=(t_1(s),t_2(s),t_3(s)):[0,L]\to \mathbb S^2,
\]
with corners parametrized at $0=s_0<s_1<\cdots<s_n=L$,
and $\bt\in C^{\infty}((s_{i-1},s_i))$ for each $i=1,\dots,n$. Each subinterval $(s_{i-1},s_i)$ generates a conical patch
\begin{equation}\label{eq:smooth_cone_param}
S_i:=\{r\,\bt(s):0\le r<1,\ s\in(s_{i-1},s_i)\}.
\end{equation}
Since each patch is generated by rays from the origin, the outward unit normal vector $\bn$ satisfies the orthogonality condition $\bt(s)\cdot \bn(s)=0$.

Given the dilation invariance of the cone, we seek the density in the separated form
\begin{equation}\label{eq:density_separation}
    \rho(\by)=r_y^\mu u(s_y),
\end{equation}
for some $\mu > -1$, where $u\in C([0,L])$. For a target point $\bx = r_x \bt(s_x) \in S_i$ and a source point $\by = r_y \bt(s_y) \in S_j$, we define the geometric coupling factors
\[
\chi(s_x,s_y):=\bt(s_x)\cdot \bt(s_y),\quad
h(s_x,s_y):=\bt(s_x)\cdot \bn(s_y).
\]
Therefore, the double-layer potential operator takes the form
\begin{align}
    \mathcal D(\rho)(\bx)
    &=
    \frac{1}{4\pi}\int_0^L h(s_x,s_y)\,u(s_y)
    \frac{\partial}{\partial \chi}
    \left(
    \int_0^1
    \frac{r_y^\mu}
    {(r_x^2+r_y^2-2r_xr_y\chi)^{1/2}}
    \,dr_y
    \right)ds_y\notag\\
    &=\frac{1}{4\pi}\int_0^L h(s_x,s_y)\,u(s_y)
    \left(
    \pi\csc(\pi\mu)P_\mu'(-\chi)\,r_x^\mu
    -
    \sum_{k=1}^\infty
    \frac{P_k'(\chi)}{k-\mu}\,r_x^k
    \right)ds_y, \label{eq:integral_ps_cones}
\end{align}
where~\eqref{eq:integral_ps_cones} is derived by applying~\eqref{eq:integral13}.

With smooth boundary data, extracting the leading singular term of order $r_x^\mu$ yields the one-dimensional integral equation governing the vertex singularity
\begin{equation}\label{eq:piecewise_smooth_system}
    \frac{\sigma}{2} u(s_x)+\mathcal K_\mu[u](s_x)=0, \quad s_x\in(s_{i-1},s_i), \quad i=1,\dots,n,
\end{equation}
where the integral operator $\mathcal K_\mu$ is defined by
\begin{equation}
    \mathcal K_\mu[u](s_x)
    :=
    \frac{1}{4\sin(\pi\mu)}
    \int_0^L h(s_x,s_y)\,P_\mu'(-\chi(s_x,s_y))\,u(s_y)\,ds_y.
\end{equation}
The singular exponents $\mu$ are precisely the values for which \eqref{eq:piecewise_smooth_system} admits a nontrivial solution $u$, thereby forming a nonlinear eigenvalue problem.

\subsection{Analysis of the kernel function}
To evaluate \eqref{eq:piecewise_smooth_system}, it is essential to analyze the behavior of the kernel function
\begin{equation}\label{eq:kernel_function}
K_{\mu}(s_x,s_y) = \frac{1}{4\sin(\pi\mu)}\,h(s_x,s_y)\,P_\mu'(-\chi(s_x,s_y))
\end{equation}
as the source point $s_y$ approaches the target point $s_x$. The Gauss hypergeometric expansions of $P_\mu(x)$ in powers of $(1-x)/2$ and $(1+x)/2$~\cite[P353, Eq.(14.3.1); P384, Eq.(15.2.2)]{NIST2010} give a rapidly convergent series expansion near $x=1$ and $x=-1$, respectively. However, near $x=-1$, due to the existence of a logarithmic singularity, we require an explicit decomposition formula, as shown in Lemma~\ref{lem:Pmu_prime_split}.

\begin{lemma}\label{lem:Pmu_prime_split}
For $\mu\in\mathbb C\setminus\mathbb Z$, we have the following decomposition near $x=-1$:
\begin{equation}\label{eq:Pmu_prime_split_exact}
P_\mu'(x) = \frac{\sin(\pi\mu)}{\pi(1+x)} - \frac{\sin(\pi\mu)}{\pi}P_\mu'(-x)\ln\!\left(\frac{1+x}{2}\right) + H_\mu(x),
\end{equation}
where $H_\mu(x)$ is an analytic function.
\end{lemma}

Now we can obtain the decomposition of $K_{\mu}(s_x,s_y)$ and its diagonal limit.

\begin{theorem}\label{thm:kernel_splitting}
Suppose $(a,b) \subset (s_{i-1},s_i)$ for some $i$. Then the kernel $K_{\mu}(s_x,s_y)$ admits the decomposition
\begin{equation}\label{eq:kernel_splitting}
K_{\mu}(s_x,s_y)=\Phi(s_x,s_y)\ln|s_x-s_y|+\Psi(s_x,s_y),\quad s_x,s_y\in(a,b),
\end{equation}
with functions $\Phi,\Psi\in C^\infty((a,b)^2)$ given by
\begin{equation}\label{eq:Phi_explicit}
\Phi(s_x,s_y)=-\frac{h(s_x,s_y)}{2\pi}
P_\mu'\bigl(\chi(s_x,s_y)\bigr),
\end{equation}
\begin{align}\label{eq:Psi_explicit}
\Psi(s_x,s_y)=&
-\frac{h(s_x,s_y)}{4\pi}\,
P_\mu'\!\bigl(\chi(s_x,s_y)\bigr)
\ln\left(\frac{1-\chi(s_x,s_y)}{2(s_x-s_y)^2}\right)\notag\\
&+\frac{h(s_x,s_y)}{4\pi\bigl(1-\chi(s_x,s_y)\bigr)}
+\frac{h(s_x,s_y)}{4\sin(\pi\mu)}\,H_\mu\!\bigl(-\chi(s_x,s_y)\bigr),
\end{align}
where $H_\mu$ is defined in Lemma~\ref{lem:Pmu_prime_split}. Moreover, the kernel admits the diagonal limit
\begin{equation}\label{eq:kernel_diagonal_limit}
\lim_{s_y\to s_x} K_{\mu}(s_x,s_y)=\frac{k_g(s_x)}{4\pi},
\end{equation}
where $k_g$ is the geodesic curvature of $\partial G$.
\end{theorem}
\begin{proof}
By Lemma~\ref{lem:Pmu_prime_split}, we can derive the decomposition of $K_{\mu}(s_x,s_y)$. Since $\bt\in C^\infty$, we have $\chi, h \in C^\infty$, and so is $\Phi$. Let $\Delta s=s_x-s_y$, then the Taylor expansion of $\bt(s_x)$ around $s_y$ gives
\[
\bt(s_x)=\bt(s_y)+\Delta s\,\bt'(s_y)+\frac{(\Delta s)^2}{2}\bt''(s_y)+O(\Delta s^3).
\]
Using the Darboux frame relation $\bt''=-\bt+k_g\bn$~\cite{doCarmo2016}, we obtain
\begin{align}
h(s_x,s_y)&=\bt(s_x)\cdot \bn(s_y)
=\frac12 k_g(s_y)(\Delta s)^2+O(\Delta s^3), \label{eq:proof_h}\\
\chi(s_x,s_y)&=\bt(s_x)\cdot \bt(s_y)
=1-\frac12(\Delta s)^2+O(\Delta s^4). \label{eq:proof_chi}
\end{align}
Hence, we have
\[
\lim _{s_y \rightarrow s_x} \frac{1-\chi(s_x,s_y)}{2(s_x-s_y)^2} = \frac{1}{4},\quad
\lim _{s_y \rightarrow s_x} \frac{h\left(s_x, s_y\right)}{4 \pi\left(1-\chi\left(s_x, s_y\right)\right)} = \frac{k_g\left(s_x\right)}{4 \pi},
\]
which implies that $\Psi$ extends continuously across the diagonal. Since $H_\mu$ is analytic and $h,\chi\in C^\infty$, it follows that $\Psi\in C^\infty$. Finally, substituting the expansions~\eqref{eq:proof_h} and~\eqref{eq:proof_chi} into~\eqref{eq:kernel_splitting}, we obtain the diagonal limit~\eqref{eq:kernel_diagonal_limit}.
\end{proof}

Unlike the cases of dihedral wedges or axisymmetric circular cones, analytically determining the singularity exponents $\mu$ for general piecewise smooth cones is challenging. In the next subsection, we propose a high-order Nystr\"om solver to compute the singularity exponents $\mu$ associated with the double-layer potential on such cones.

\subsection{An eigenvalue solver for piecewise smooth cones}\label{subsec:NEP_solver}
Our starting point is the one-dimensional integral equation
\eqref{eq:piecewise_smooth_system}. After discretization, this equation becomes a finite-dimensional nonlinear eigenvalue problem with respect to $\mu$, where two numerical features require attention. First, the angular eigenfunction $u(s)$ may have algebraic singularities near the corner points $s_0,s_1,\ldots,s_n$, reflecting the edge singularities generated by adjacent conical patches. Second, as shown in Theorem~\ref{thm:kernel_splitting}, the self-interaction kernel on each patch contains a logarithmic branch. To address these challenges, we propose an approach that combines the grading technique~\cite{baohualai25} near the corners with logarithmically corrected trapezoidal rules~\cite{Kress2014} for the self-interaction blocks.

To resolve the endpoint behavior, we introduce sigmoidal transformations~\cite{Elliott1998} $s = \eta_{p,i}(\tau) = s_{i-1} + (s_i-s_{i-1})\,g_p(\tau)$, $\tau\in[0,1]$, where $g_p$ satisfies $g_p(0)=0$, $g_p(1)=1$ and has a Jacobian $\eta_{p,i}'(\tau)$ that vanishes at the endpoints. As shown in Table~\ref{tab:graded_mappings}, we choose the algebraic-power and integrated-sine gradings in this work. Other grading techniques, including the Sidi transformation, the Sag--Szekeres transformation, and the Mori transformations, can also be used for the present problems with lower efficiency.

\begin{table}[H]
\centering
\caption{Two sigmoidal transformations $g_{p}(\tau)$ used in our experiments.}
\label{tab:graded_mappings}
\begin{tabular}{c c c}
\hline\noalign{\vskip2pt}
Type & $g_p(\tau)$ & $p$ \\
\noalign{\vskip2pt}\hline\noalign{\vskip4pt}
Algebraic-power & $\displaystyle \frac{\tau^p}{(1-\tau)^p+\tau^p}$ & $p\propto\sqrt{N}$ \\
\noalign{\vskip2pt}\hline\noalign{\vskip4pt}
Integrated-sine &
$\displaystyle\frac{\int_0^\tau(\tau-t)\sin^{p-2}(2\pi t)dt}{\int_0^1(1-t)\sin^{p-2}(2\pi t)dt}$
& $p = 11$ \\
\noalign{\vskip2pt}\hline
\end{tabular}
\end{table}

Let $\tau_j=j/N_i$, $j=0,1,\dots,N_i-1$, be uniformly spaced nodes on $[0,1)$
for the $i$-th subinterval, and set $s_{i,j}=\eta_{p,i}(\tau_j)$,
$J_{i,j}=\eta_{p,i}'(\tau_j)$. For a target point $s_{i,k}$ on the $i$-th subinterval, after a change of variables, the integral operator becomes
\begin{equation}\label{eq:Kmu_after_change_var}
    \mathcal K_\mu[u](s_{i,k}) =
    \sum_{\ell=1}^n \int_0^1
    K_\mu(s_{i,k},\eta_{p,\ell}(\tau))\,
    u(\eta_{p,\ell}(\tau))\,
    \eta_{p,\ell}'(\tau)\,d\tau .
\end{equation}
For $i\neq \ell$, the corresponding kernel is smooth away from the endpoints, and the ordinary trapezoidal rule applies directly:
\begin{equation}\label{eq:off_diagonal_block}
    \left(\bK_\mu^{(i,\ell)}\right)_{k,j}
    = \frac{1}{N_\ell}\,K_\mu(s_{i,k},s_{\ell,j})\,J_{\ell,j}.
\end{equation}

It remains to consider the self-interaction block $i=\ell$. By Theorem~\ref{thm:kernel_splitting}, the kernel has the local form
\[
    K_\mu(s_x,s_y)
    = \Phi(s_x,s_y)\ln|s_x-s_y|+\Psi(s_x,s_y),
\]
where $\Phi$ and $\Psi$ are smooth on each smooth subinterval. Under the graded change of variables, the logarithmic term can be written as
\[
    \ln|\eta_{p,i}(\tau_k)-\eta_{p,i}(\tau)| = \frac12\ln\!\big(4\sin^2(\pi(\tau_k-\tau))\big) + \ln\left|\frac{\eta_{p,i}(\tau_k)-\eta_{p,i}(\tau)}{2\sin(\pi(\tau_k-\tau))}\right|.
\]
The second term is smooth and can be absorbed into the smooth remainder. Hence, for the self-interaction block, we write
\[
    K_\mu(s_{i,k},\eta_{p,i}(\tau))\,\eta_{p,i}'(\tau) = \widetilde\Phi(\tau_k,\tau)\ln\left(4\sin^2(\pi(\tau_k-\tau))\right) + \widetilde\Psi(\tau_k,\tau),
\]
where $\widetilde\Phi$ and $\widetilde\Psi$ are smooth away from the endpoints.

For even $N_i$, we use the corrected trapezoidal rule
\begin{equation}\label{eq:kress_rule}
    \int_0^1\widetilde\Phi(\tau_k,\tau)\ln\!\big(4\sin^2(\pi(\tau_k-\tau))\big)\,d\tau \approx \sum_{j=0}^{N_i-1}R_{k,j}^{(N_i)}\widetilde\Phi(\tau_k,\tau_j),
\end{equation}
where
\begin{equation}\label{eq:kress_weights}
    R_{k,j}^{(N_i)}=-\frac{2}{N_i}\sum_{m=1}^{N_i/2-1}\frac{\cos(2\pi m(\tau_k-\tau_j))}{m}-\frac{2}{N_i^2}\cos(N_i\pi(\tau_k-\tau_j)).
\end{equation}
The smooth remainder is discretized by the ordinary trapezoidal rule. Hence, the self-interaction block is approximated by
\begin{equation}\label{eq:self_block_kress}
    \left(\bK_\mu^{(i,i)}\right)_{k,j}
    =
    \frac{1}{N_i}\widetilde\Psi(\tau_k,\tau_j)
    +
    R_{k,j}^{(N_i)}\widetilde\Phi(\tau_k,\tau_j).
\end{equation}
Specifically, when $s_{i,k}=s_{i,j}$, we apply the diagonal limit derived in~\eqref{eq:kernel_diagonal_limit}.

The discretized nonlinear eigenvalue problem is therefore
\begin{equation}\label{eq:nystrom_matrix}
    \bA(\mu)\,\bu = \left[\frac{\sigma}{2}\bI+\bK_\mu^{(N)}\right] \bu =\mathbf{0},
\end{equation}
where $\bK_\mu^{(N)}$ denotes the logarithmically corrected trapezoidal discretization of $\mathcal K_\mu$, and $\bu$ denotes the vector of all nodal values of the angular eigenfunction. The singularity exponents $\mu$ are the roots of the nonlinear equation
\begin{equation}\label{eq:det_objective}
    f(\mu) :=\det\bA(\mu)=0.
\end{equation}
After choosing a bracketing interval, the desired exponent $\mu_\ast$ can be refined to machine precision using derivative-free algorithms. In this paper, we adopt Brent's method~\cite{Brent1973}, which combines bisection, secant, and inverse quadratic interpolation.

\begin{remark}[Constant-density subtraction]\label{rem:subtraction}
To address the inconsistency of the jump relation near the conical corners, we introduce a subtraction technique, where the original integral operator is split into
\begin{equation}\label{eq:subtraction_technique}
    \mathcal K_\mu[u](s_x) = (\mathcal K_\mu-\mathcal K_0)[u](s_x)+ \mathcal K_0[u-u(s_x)](s_x)+ \mathcal K_0[1](s_x)u(s_x).
\end{equation}
As shown in~\ref{app:subtraction_constant}, we have
\begin{equation}
    \frac12+\mathcal K_0[1](s_x) = \frac{|G|}{4\pi},
\end{equation}
where $G \subset \mathbb S^2$ denotes the spherical cross-section of the cone and $|G|$ denotes its area. Replacing the constant-density contribution by its exact value, we can obtain
\[
\mathcal K_\mu[u](s_x) \approx \widehat{\mathcal K}_\mu[u](s_x) - \widehat{\mathcal K}_0[1](s_x)\,u(s_x) + \left(\frac{|G|}{4\pi}-\frac{1}{2}\right)u(s_x),
\]
where $\widehat{\,\cdot\,}$ denotes the numerical evaluation. This trick improves both the accuracy and the convergence rate of the quadrature rule.
\end{remark}

\section{Numerical experiments}\label{sec:experiments}
We now present numerical experiments to validate the solver developed above and to illustrate how the local edge and vertex mechanisms appear in conical geometries. The examples start from a trihedral benchmark to regular $m$-hedral cones approaching the circular-cone limit, and finally to non-polyhedral droplet-type cones with smooth geometric perturbations and varying opening angles.

\subsection{Example 1: Interactions between edge and vertex singularities}
On polyhedral domains, the interaction between edge and vertex singularities is governed by a tensor-product decomposition~\cite{Petersdorff1990}, under which the singular part of the solution behaves as
\[
u_{\mathrm{sing}}(r, \theta) \approx r^{\lambda} \left( \sum_{m} c_m \theta^{\alpha_m} + w(\theta) \right),
\]
where $\lambda$ is the vertex exponent, $\alpha_m$ the 2D edge exponents, $\theta$ the angular distance to the edge, and $w(\theta)$ a smoother remainder.

In this subsection, we focus on the right-angled trihedral cone $S=S_1\cup S_2\cup S_3$ in the first octant, as depicted in Figure~\ref{fig:trihedral}, where
\begin{align}
    S_1 &= \{r\,\bt_1(\theta): r \in (0, 1), \, 0<\theta<\pi/2\}, \quad \text{where } \bt_1(\theta) = (0, \cos\theta, \sin\theta)^T, \notag\\
    S_2 &= \{r\,\bt_2(\theta): r \in (0, 1), \, 0<\theta<\pi/2\}, \quad \text{where } \bt_2(\theta) = (\sin\theta, 0, \cos\theta)^T, \label{eq:trihedral_faces}\\
    S_3 &= \{r\,\bt_3(\theta): r \in (0, 1), \, 0<\theta<\pi/2\}, \quad \text{where } \bt_3(\theta) = (\cos\theta, \sin\theta, 0)^T. \notag
\end{align}
\begin{figure}[!htbp]
    \centering
    \includegraphics[width=0.3\textwidth]{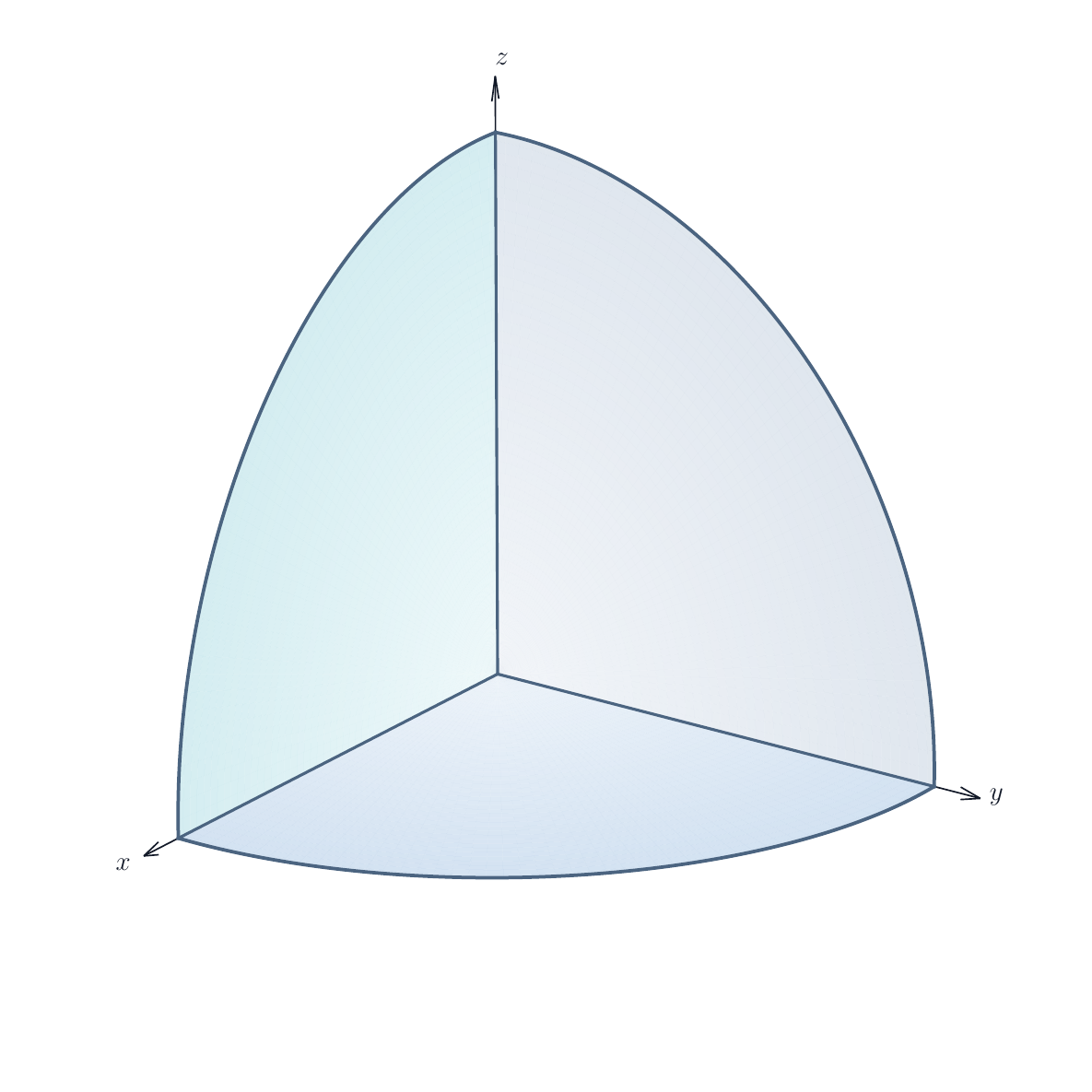}
    \caption{A trihedral cone with right-angled corner.}
    \label{fig:trihedral}
\end{figure}
The significance of this model lies in its geometric connection to the Fichera corner~\cite{pester2005cocos}, a classical benchmark domain for studying three-dimensional non-convex corner singularities. The Fichera corner is formed by removing the first octant from a cube, defined as the set $(-1, 1)^3 \setminus [0,1]^3$. Therefore, analyzing the leading corner singularity of the Fichera domain is equivalent to solving the exterior Dirichlet eigenvalue problem associated with $S$.

Due to the cyclic symmetry of the geometry, each face has opening angle $\pi/2$ and the angular eigenfunctions on all three faces are identical: $u_1 = u_2 = u_3$. Choosing target points on $S_3$, where $h_{31}(\theta_x) = -\cos\theta_x$ and $h_{32}(\theta_x) = -\sin\theta_x$, and summing the contributions from $S_1$ and $S_2$, we obtain a reduced $N_i \times N_i$ matrix $\mathbf{M}(\mu)$ defined by
\begin{equation}\label{eq:reduced_nystrom}
    M_{p,l}(\mu) = \frac{1}{2} \delta_{p,l} + W_l K(\theta_p, \theta_l; \mu), \quad p, l = 1, \dots, N_i,
\end{equation}
where $W_l$ are the corresponding weights and the kernel $K$ is given by
\begin{equation}
    K(\theta_p, \theta_l; \mu) = -\frac{1}{4\sin(\pi\mu)} \Big( \cos\theta_p P_\mu'\big(-\sin\theta_p \cos\theta_l\big) + \sin\theta_p P_\mu'\big(-\cos\theta_p \sin\theta_l\big) \Big).
\end{equation}

For the right-angled trihedral cone, the benchmark value of the leading exponent is $\mu_{\rm ref}=0.45417371533061$ as reported by IMCS~\cite{IMCSFicheraCorner}. We compute the leading exponent with the number of quadrature nodes $N=2^4,\,2\cdot 2^4,\,\ldots,\,16\cdot 2^4$ using algebraic-power grading. As shown in Figure~\ref{fig:Exp1_cube}(a), the self-convergence indicator $|\mu_N-\mu_\ast|/\mu_\ast$ decays root-exponentially and reaches round-off, with a least-squares fit over the first five levels giving $|\mu_N-\mu_\ast|/\mu_\ast \approx 3.4\exp(-3.28\sqrt{N})$. The reported benchmark appears accurate to about eight digits.
\begin{figure}[!htbp]
    \centering
    \includegraphics[width=0.85\textwidth]{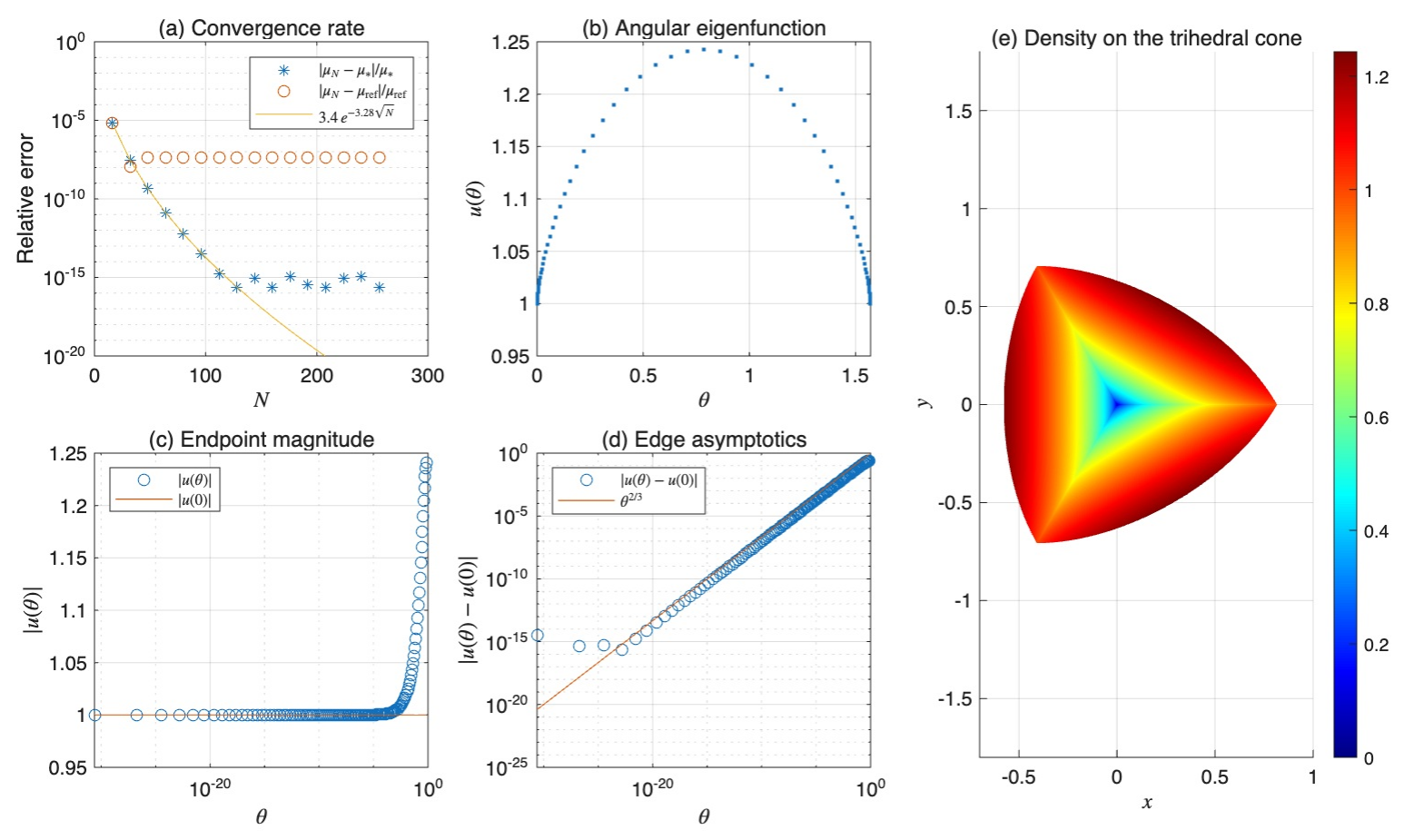}
    \caption{Numerical results for the exterior right-angled trihedral cone: (a) self-convergence of the leading exponent; (b) angular eigenfunction $u(\theta)$ on one face; (c) endpoint magnitude of $u(\theta)$; (d) log--log plot of $|u(\theta)-u(0)|$; (e) singular function $r^{\mu_\ast}u(\theta)$.}
    \label{fig:Exp1_cube}
\end{figure}
Once $\mu_\ast \approx 0.45417373430136$ is found, the discrete angular eigenfunction is recovered from the null vector of the matrix. Its endpoint behavior reflects the transition toward the two-dimensional edge singularities of Subsection~\ref{subsec:CaseI}. By Theorem~\ref{thm:Dihedral_wedge}, the exterior opening angle $\alpha=3\pi/2$ yields the edge exponent $\mu_{\text{dihedral}}=2/3$. Figure~\ref{fig:Exp1_cube}(b)--(d) show $u(\theta)$, its endpoint magnitude, and the log--log plot of $|u(\theta)-u(0)|$, which matches the reference slope $2/3$, confirming that the endpoint behavior is controlled exactly by $\mu_{\text{dihedral}}$. Figure~\ref{fig:Exp1_cube}(e) visualizes the reconstructed singular function $r^{\mu_\ast}u(\theta)$, illustrating the interaction between the vertex exponent $\mu_\ast$ and the edge behavior of $u(\theta)$.

\subsection{Example 2: Asymptotics of $m$-hedral cones toward circular cones}
In this example, we study the asymptotic behavior of the leading singular exponents of regular $m$-hedral cones as $m\rightarrow\infty$. It tests the applicability of our general dilation-invariant conical formulation and its consistency with the circular cone.

For that purpose, we construct a sequence of regular $m$-hedral cones inscribed in a circular cone with a fixed half-apex angle $\alpha$. The $m$ edges of the inscribed polyhedral cone lie exactly on the surface of the circular cone and are defined by the unit directional vectors $\bv_k = (\sin\alpha \cos\varphi_k, \sin\alpha \sin\varphi_k, \cos\alpha)^\top$, where $\varphi_k = 2\pi (k-1)/m$, $k = 1, \dots, m$. Each planar face $S_i$ is a sector spanned by the adjacent vectors $\bv_i$ and $\bv_{i+1}$ with $\bv_{m+1} \equiv \bv_1$. Since the bounding rays are uniformly distributed, every face has the same opening angle $\omega_m$ determined by the inner product $\cos\omega_m = \bv_i \cdot \bv_{i+1} = \sin^2\alpha \cos\left(2\pi/m\right) + \cos^2\alpha$.

Within our boundary integral framework, we construct a local orthonormal basis $\{\be_{i,1}, \be_{i,2}\}$ on each face $S_i$. Setting $\be_{i,1} = \bv_i$, we apply Gram--Schmidt orthogonalization to obtain the second basis vector $\be_{i,2} = \left(\bv_{i+1} - \cos\omega_m \bv_i\right)/\sin\omega_m$. This yields the local polar parameterization $\bt_i(\theta) = \cos\theta \, \be_{i,1} + \sin\theta \, \be_{i,2}$, where $\theta \in (0, \omega_m)$. Moreover, the outward unit normal on $S_i$ can be represented by $\bn_i = \left(\bv_{i+1} \times \bv_i\right) / \sin\omega_m$. Substituting these parameterizations, we obtain explicit expressions for the geometric coupling factors $\chi_{ij}(\theta_x, \theta_y) = \bt_i(\theta_x) \cdot \bt_j(\theta_y)$ and $h_{ij}(\theta_x) = \bt_i(\theta_x) \cdot \bn_j$. We then use the nonlinear eigenvalue solver developed in Section~\ref{subsec:NEP_solver} to compute the leading singularity exponent $\mu_\ast$.

Specifically, we fix the half-apex angle at $\alpha=\pi/6$ and consider the exterior problem. For $m=3,\dots,30$, each face is discretized by the Nystr\"om method with the number of quadrature nodes per face $N=2^4,\;2\cdot 2^4,\;\ldots,\;16\cdot 2^4$ together with algebraic-power grading. As a reference, the corresponding circular-cone exponent for $\alpha=\pi/6$ by~\eqref{eq:V_leading} is $\mu_{\mathrm{circ}}(\pi/6) \approx 0.34618394064801$. For the self-convergence test in Figure~\ref{fig:Exp2_selfconvergence}, the relative error against the finest value decays root-exponentially for each fixed $m$.
\begin{figure}[!htbp]
    \centering
    \includegraphics[width=0.6\textwidth]{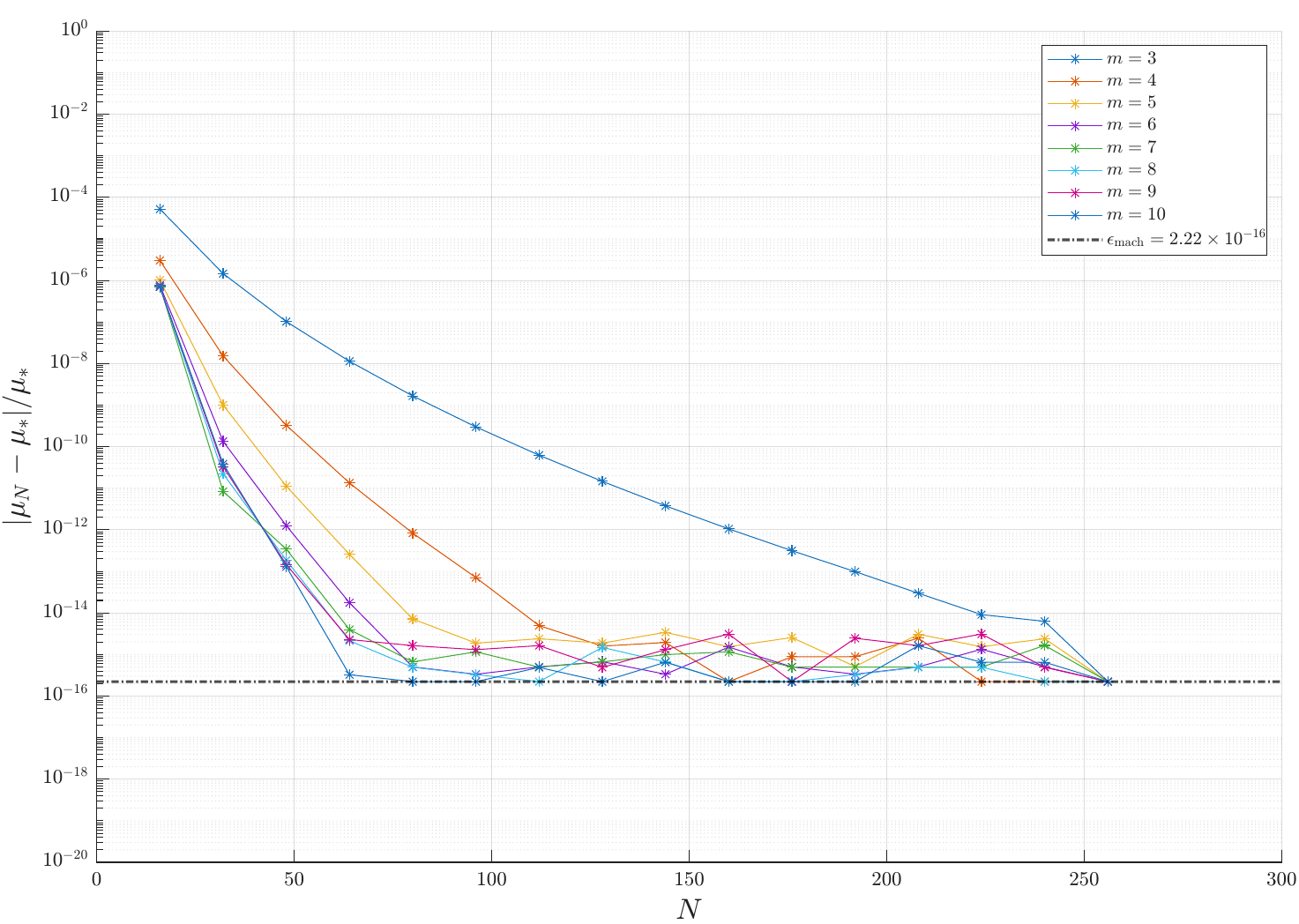}
    \caption{Self-convergence for exterior regular $m$-hedral cones.}
    \label{fig:Exp2_selfconvergence}
\end{figure}
Moreover, Figure~\ref{fig:Experiment2}(a) plots $\mu_\ast(m)$ together with the reference line $\mu_{\mathrm{circ}}(\pi/6)$, where the sequence increases monotonically and approaches the circular-cone value from below. In Figure~\ref{fig:Experiment2}(b), we fit the relative error by $\varepsilon_{\mathrm{rel}}(m)\approx 1.38\,m^{-1.935}$, close to the $O(m^{-2})$ rate expected from the second-order accuracy of an inscribed regular polygon approximating a circle. This confirms the asymptotic consistency of our general formulation with the circular cone.
\begin{figure}[htbp]
    \centering
    \begin{minipage}[t]{0.42\textwidth}
        \centering
        \includegraphics[width=\linewidth]{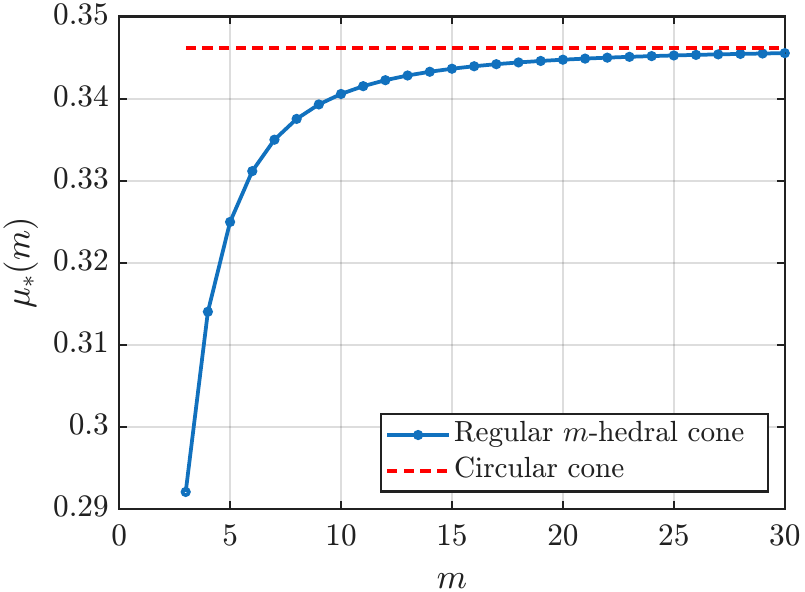}\\
        \small (a) Numerical values of $\mu_\ast(m)$
    \end{minipage}\hspace{0.04\textwidth}
    \begin{minipage}[t]{0.42\textwidth}
        \centering
        \includegraphics[width=\linewidth]{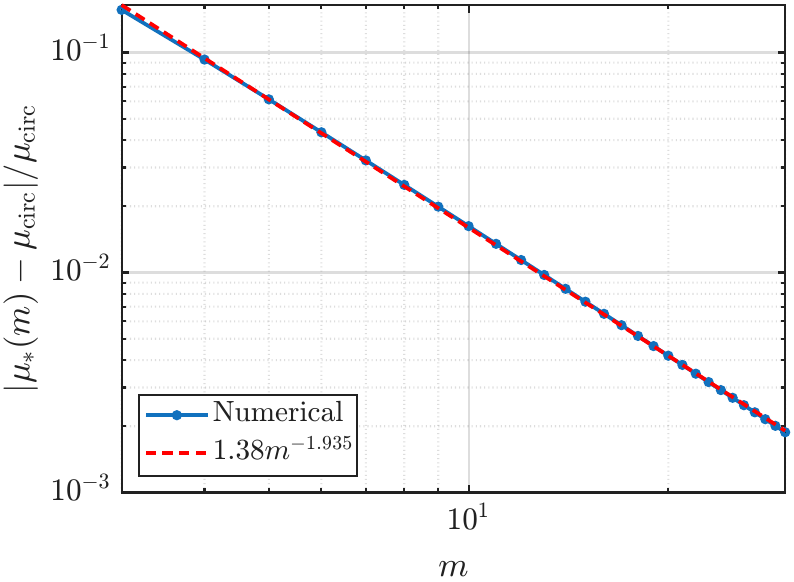}\\
        \small (b) Relative error and algebraic fit
    \end{minipage}
    \caption{Asymptotic convergence of the leading singular exponent $\mu_\ast(m)$ for regular $m$-hedral cones inscribed in the circular cone with half-apex angle $\alpha=\pi/6$: (a) Numerical values of $\mu_\ast(m)$ versus $m$; (b) Relative error $\varepsilon_{\mathrm{rel}}(m)$ on a log--log scale.}
    \label{fig:Experiment2}
\end{figure}

\subsection{Example 3: Non-polyhedral droplet-type cones}
We finally consider a class of non-polyhedral dilation-invariant cones generated by smooth droplet-type planar curves. Let $\zeta(s)=f(s)\exp(i\theta(s-\tfrac12))$, $0\le s\le 1$,
and define the spherical cross-section by
\[
    \bt(s)=\frac{(\Re\zeta(s),\Im\zeta(s),1)^T}{\sqrt{1+|\zeta(s)|^2}},
\]
so that the corresponding conical surface is $S=\{r\bt(s):0<r<1,\;0<s<1\}$. In this example, we consider the basic droplet curve $f_1(s)=\sin(\pi s)$
and a wavy droplet curve $f_2(s)=\sin(\pi s)\bigl(1+0.15\cos(10\pi s)\bigr)$, both with $\theta=2\pi/3$, where the inherited endpoint exponent is $\mu_{\text{dihedral}}=3/4$. An illustration of the geometries is shown in Figure \ref{fig:Exp3_droplet}(d).

Specifically, for the basic droplet, we use the algebraic-power grading, whereas for the wavy droplet, whose generating curve exhibits finer oscillations, we use the integrated-sine grading. Following Remark~\ref{rem:subtraction}, for the present droplet parametrization, the normalized solid angle evaluates to
\[
\frac{|G|}{4\pi}=\frac{\theta}{4\pi}\int_0^1\Big(1-\frac{1}{\sqrt{1+f(s)^2}}\Big)ds.
\]
Figure~\ref{fig:Exp3_droplet} shows the results for both shapes. Panel~(a) shows the fast self-convergence of the leading exponent: root-exponential for the basic droplet and exponential for the wavy droplet, both reaching machine precision rapidly, confirming the effectiveness of the discretization for general smooth dilation-invariant cones. Panels~(b) and (c) show the angular eigenfunction and its endpoint behavior, whose log--log plots of $|u(s)-u(0)|$ match the reference slope $3/4$. Finally, Panel~(d) reconstructs the singular functions $r^{\mu_\ast}u(s)$ on the two dilation-invariant cones.
\begin{figure}[!htbp]
    \centering
    \begin{minipage}[t]{0.48\textwidth}
        \centering
        \includegraphics[width=\linewidth]{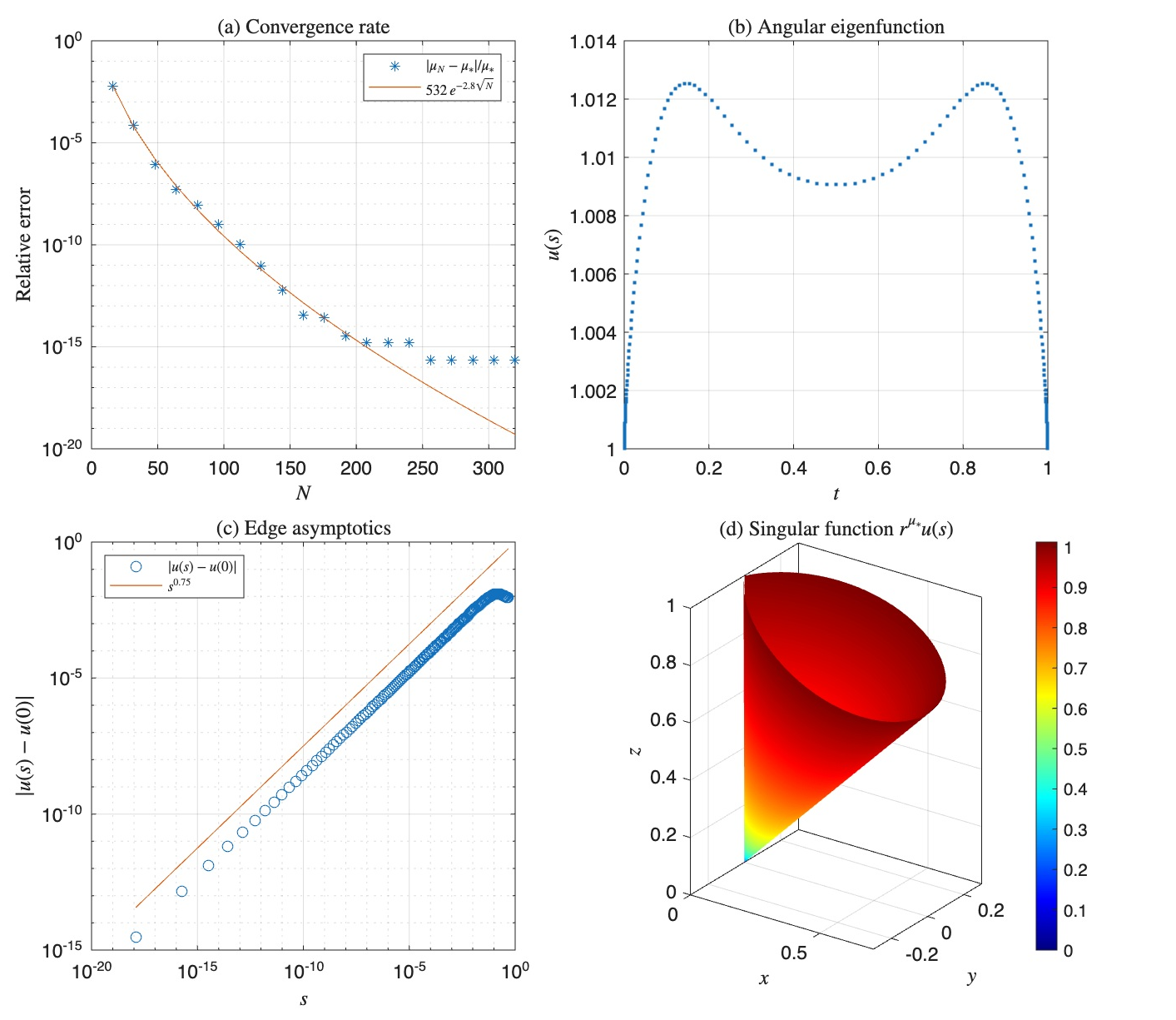}\\
        \small (A) Basic droplet
    \end{minipage}
    \hfill
    \begin{minipage}[t]{0.48\textwidth}
        \centering
        \includegraphics[width=\linewidth]{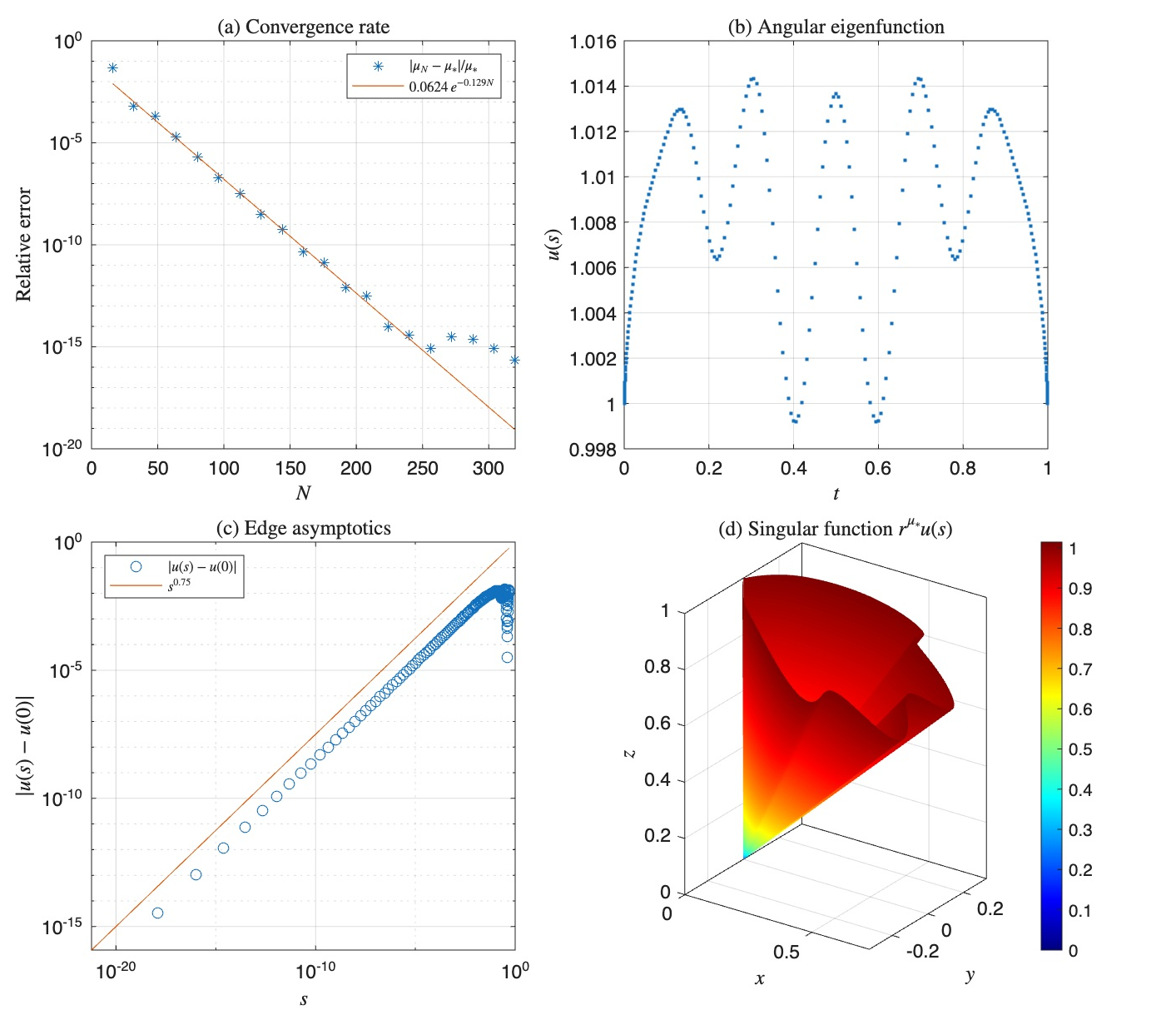}\\
        \small (B) Wavy droplet
    \end{minipage}
    \caption{Numerical results for (A) the basic droplet cone and (B) the wavy droplet cone. For each cone, panel~(a) shows the self-convergence of the leading singular exponent, and panels~(b)--(d) show the angular eigenfunction, its endpoint behavior and the reconstructed singular functions.}
    \label{fig:Exp3_droplet}
\end{figure}

To assess the dependence on the opening angle, we repeat the computation for several values of $\theta$. The theoretical endpoint exponents $\mu_{\text{dihedral}}$, fitted self-convergence laws, the smallest $N$ at which machine precision is reached, and finest leading exponents are summarized in Table~\ref{tab:Exp3_droplet}.

\begin{table}[H]
\centering
\caption{Self-convergence laws and fitted leading singular exponents for several opening angles.}
\label{tab:Exp3_droplet}
\footnotesize
\setlength{\tabcolsep}{3pt}
\begin{tabular}{c c c c c c}
\hline
$\theta$ & $\mu_{\text{dihedral}}$ & Geometry & Self-convergence law & $N_{\rm mach}$ & $\mu_\ast$\\
\hline
\multirow{2}{*}{$\pi/4$} & \multirow{2}{*}{$4/7$} & Basic droplet & $4.95\exp(-1.96\sqrt{N})$ & $352$ & $0.233245724690681$ \\
        & & Wavy droplet & $0.00711\exp(-0.0503 N)$ & $544$ & $0.238743996027733$ \\
\multirow{2}{*}{$\pi/2$} & \multirow{2}{*}{$2/3$} & Basic droplet & $519\exp(-2.85\sqrt{N})$ & $192$ & $0.255174383123647$ \\
        & & Wavy droplet & $0.0286\exp(-0.105 N)$ & $256$ & $0.260342798451775$ \\
\multirow{2}{*}{$3\pi/4$} & \multirow{2}{*}{$4/5$} & Basic droplet & $1.51\times10^{3}\exp(-2.94\sqrt{N})$ & $192$ & $0.277266520377783$ \\
        & & Wavy droplet & $0.0706\exp(-0.14 N)$ & $224$ & $0.282323344084661$ \\
\multirow{2}{*}{$5\pi/4$} & \multirow{2}{*}{$4/5$} & Basic droplet & $990\exp(-2.94\sqrt{N})$ & $192$ & $0.320717998562855$ \\
        & & Wavy droplet & $0.168\exp(-0.182 N)$ & $192$ & $0.325862977275596$ \\
\multirow{2}{*}{$3\pi/2$} & \multirow{2}{*}{$2/3$} & Basic droplet & $511\exp(-2.93\sqrt{N})$ & $192$ & $0.341093466874671$ \\
        & & Wavy droplet & $0.329\exp(-0.202 N)$ & $160$ & $0.346357978077202$ \\
\multirow{2}{*}{$7\pi/4$} & \multirow{2}{*}{$4/7$} & Basic droplet & $222\exp(-2.95\sqrt{N})$ & $160$ & $0.359808910849466$ \\
        & & Wavy droplet & $0.126\exp(-0.199 N)$ & $160$ & $0.365228519182439$ \\
\hline
\end{tabular}
\end{table}

\section{Conclusion}\label{sec:conclusions}
This work presents a Nystr\"{o}m-based framework for computing singular exponents of elliptic boundary value problems on three-dimensional nonsmooth domains. We first derived the two dihedral edge families and the associated truncated coefficient expansion on the wedge. For axisymmetric circular cones, we characterized the admissible modewise singular exponents and obtained an explicit truncated coefficient mapping for each Fourier mode. For general dilation-invariant cones, where closed-form formulas are no longer available, we formulated the singularity exponents through a boundary integral equation on the spherical cross-section and developed a Nystr\"{o}m-based nonlinear eigenvalue solver. The numerical experiments confirmed the root-exponential (and exponential for the integrated-sine grading) self-convergence rate, the predicted edge exponent $2/3$ on the trihedral cone, the $O(m^{-2})$ convergence of $m$-hedral cones to the circular limit, and the robustness of our framework on non-polyhedral droplet cones.

Although a general cornered domain is not globally dilation-invariant, its leading singular behavior is expected to be governed by the associated tangent cone. This motivates the focus on conical geometries in this work, as they determine the local singular exponents of the boundary density and provide canonical model problems for understanding the interaction between edge and vertex singularities. The present framework therefore provides a foundation for future analysis of more general three-dimensional piecewise smooth domains. Further work includes incorporating the computed singular functions into the design of high-order boundary integral solvers for nonsmooth geometries and extending the analysis to wave scattering problems.

\section*{Acknowledgements}
JL was partially supported by the National Key Research and Development Program of China (2025YFA1016800), NSFC general program (No.~12371427) and the ``Xiaomi Young Scholars'' program from the Xiaomi Foundation.

\appendix
\renewcommand{\thetheorem}{\Alph{section}.\arabic{theorem}}

\section{Properties of Legendre functions}\label{app:legendre}
The analysis of circular and general piecewise smooth cones relies heavily on the properties of Legendre functions. For the reader's convenience, this section collects the formulations used in the main text. More details can be found in~\cite{NIST2010,Bateman1954}.

For $\mu\in\mathbb{R}$, let $P_\mu$ denote the Legendre function of the first kind of degree $\mu$ satisfying the Legendre differential equation
\[
\frac{d}{d t}\left[\left(1-t^2\right) \frac{d f}{d t}\right]+\mu(\mu+1) f=0, \quad t\in (-1,1).
\]
Let $P_\mu^m$ denote the associated Legendre function of order $m$. For nonnegative integer orders $m\in\mathbb{N}_0$, we use the convention
\[
P_\mu^m(t):=\left(1-t^2\right)^{m / 2} \frac{d^m P_\mu(t)}{d t^m}, \quad m=0,1,\ldots.
\]
For negative integer orders, the functions are defined by the connection formula
\[
P_\mu^{-m}(t)
:=(-1)^{m}\frac{\Gamma(\mu-m+1)}{\Gamma(\mu+m+1)}P_\mu^m(t),
\qquad m=0,1,2,\ldots.
\]
With these conventions, $P_\mu^m$ satisfies the associated Legendre differential equation
\[
\frac{d}{d t}\left[\left(1-t^2\right) \frac{d f}{d t}\right]+\left[\mu(\mu+1)-\frac{m^2}{1-t^2}\right] f=0, \quad t\in (-1,1).
\]

\begin{lemma}[Series representations of $P_\mu$]\label{app:lem:gauss_series}
	For $\mu \in \mathbb{C}$, $P_{\mu}(x)$ admits the Gauss hypergeometric series expansion in powers of $(1-x)/2$:
    \begin{equation}\label{app:eq:Gauss_series}
        P_\mu(x)=\sum_{n=0}^{\infty} \frac{(-\mu)_n(\mu+1)_n}{(n!)^2}\left(\frac{1-x}{2}\right)^n,\quad x\in(-1,1),
    \end{equation}
	where $(a)_n = \Gamma(a+n)/\Gamma(a)$ is the Pochhammer symbol. For $\mu \in \mathbb{C} \backslash \mathbb{Z}$, the logarithmic expansion of $P_\mu(x)$ in powers of $(1+x)/2$ is
    \begin{align}
        P_\mu(x)= & \frac{\sin \mu \pi}{\pi} \sum_{n=0}^{\infty} \frac{(-\mu)_n(\mu+1)_n}{(n!)^2}\left(\frac{1+x}{2}\right)^n \notag\\
        & \times\left[\psi(-\mu+n)+\psi(\mu+1+n)-2 \psi(1+n)+\ln \frac{1+x}{2}\right], \quad x\in(-1,1),\label{app:eq:Gauss_series_P}
    \end{align}
    where $\psi$ is the digamma function.
\end{lemma}

\begin{lemma}[Addition theorem]\label{app:lem:add}
	For $\theta_1, \theta_2, \theta_1+\theta_2 \in [0,\pi)$ and $\phi \in \mathbb{R}$, we have
	\begin{align}
		&P_\mu\left(\cos \theta_1 \cos \theta_2+\sin \theta_1 \sin \theta_2 \cos \phi\right) \notag\\
			=\, &P_\mu\left(\cos \theta_1\right) P_\mu\left(\cos \theta_2\right)+2 \sum_{m=1}^{\infty}(-1)^m P_\mu^{-m}\left(\cos \theta_1\right) P_\mu^m\left(\cos \theta_2\right) \cos (m \phi). \label{app:eq:addition_theorem}
	\end{align}
\end{lemma}

\begin{lemma}[Generating function]\label{app:lem:generating_function}
    The generating function of $P_n(x)$ is
    \begin{equation}\label{app:eq:generating_function}
        \frac{1}{\sqrt{1-2 x t + t^2}} = \sum_{n=0}^{\infty} P_n(x)t^n,  \quad |t|<1,\ |x|\leq 1.
    \end{equation}
\end{lemma}

\begin{lemma}[Wronskian of Legendre functions]\label{app:lem:wron}
	The Wronskian of $P_{\mu}^{-m}(x)$ and $P_{\mu}^{-m}(-x)$ is
		\begin{equation}\label{app:eq:wronskian}
		\mathscr{W} \left\{P_\mu^{-m}(x), P_\mu^{-m}(-x)\right\}  =\frac{2}{\Gamma(m-\mu) \Gamma(\mu+m+1)\left(1-x^2\right)}.
	\end{equation}
\end{lemma}

\section{Proof of the dihedral coefficient expansion}\label{app:dihedral_wedge}
For the interior problem on the dihedral wedge, substituting $\rho_1=r^\mu$ and $\rho_2=\eta r^\mu$ gives
\[
\begin{cases}
\displaystyle -\frac12-\eta\,\frac{\sin\big((\pi-\alpha)\mu\big)}{2\sin(\pi\mu)}=0,\\[2mm]
\displaystyle -\frac12\,\eta-\frac{\sin\big((\pi-\alpha)\mu\big)}{2\sin(\pi\mu)}=0.
\end{cases}
\]
Eliminating $\eta$ yields
\begin{equation}\label{app:eq:transcendental_edge}
	E_{\mu}(\alpha) = \frac{\sin\big((\pi-\alpha)\mu\big)}{\sin(\pi\mu)}\pm 1 = 0.
\end{equation}
The choices $\eta=1$ and $\eta=-1$ correspond to the symmetric and antisymmetric modes. Sum-to-product identities give, for $\eta=1$,
\[
\sin((\pi-\alpha)\mu)+\sin(\pi\mu)
=2\sin\left(\pi\mu-\frac{\alpha\mu}{2}\right)
\cos\left(-\frac{\alpha\mu}{2}\right)=0,
\]
so $\mu=(2k+1)\pi/\alpha$ or $\mu=2k\pi/(2\pi-\alpha)$, $k\in\mathbb N_0$. For $\eta=-1$, one obtains $\mu=2k\pi/\alpha$ or $\mu=(2k+1)\pi/(2\pi-\alpha)$. The exterior problem only interchanges the symmetric and antisymmetric roles, and the exponent set is unchanged.

To match the notation of Serkh and Rokhlin~\cite{serkh2016solution}, fix a cross-section orthogonal to the edge and introduce a signed radial coordinate $t$, with $t>0$ on $S_1$ and $t<0$ on $S_2$. For $r=|t|$ define
\[
\rho_{\mathrm{even}}(r):=\frac12\bigl(\rho_1(r)+\rho_2(r)\bigr),
\quad
\rho_{\mathrm{odd}}(r):=\frac12\bigl(\rho_1(r)-\rho_2(r)\bigr).
\]
These are the symmetric and antisymmetric modes above. Theorems~5.4 and~5.5 in~\cite{serkh2016solution} then imply the following restatement, which in turn yields Theorem~\ref{thm:Dihedral_wedge}. The finitely many opening angles excluded below are precisely those treated in~\cite{serkh2019detailed}, where the expansion is extended to all $\alpha\in(0,2\pi)$ at the cost of including terms of the form $r^{\mu}\ln r$.

\begin{lemma}\label{app:lem:2D_map}
    For all but finitely many $\alpha \in (0,2\pi)$, suppose that the even and odd parts of the induced cross-sectional boundary data can be represented by convergent Taylor series. Then, for every positive integer $n$, there exist unique coefficients $b_0,\dots,b_n$ and $c_0,\dots,c_n$ such that
    \[
    \rho_1^{(n)}(r)=\rho_{\mathrm{even}}^{(n)}(r)+\rho_{\mathrm{odd}}^{(n)}(r),
    \quad
    \rho_2^{(n)}(r)=\rho_{\mathrm{even}}^{(n)}(r)-\rho_{\mathrm{odd}}^{(n)}(r)
    \]
    solve the boundary integral equation up to a residual of order $O(r^{n+1})$, where
    \begin{align*}
    \rho_{\mathrm{even}}^{(n)}(r)
    =& \sum_{k=0}^{\lceil n/2\rceil-1} b_{2k+1} r^{(2k+1)\pi/\alpha}
    + \sum_{k=0}^{\lfloor n/2\rfloor} c_{2k} r^{2k\pi/(2\pi-\alpha)}, \\
    \rho_{\mathrm{odd}}^{(n)}(r)
    =& \sum_{k=0}^{\lfloor n/2\rfloor} b_{2k} r^{2k\pi/\alpha}
    + \sum_{k=0}^{\lceil n/2\rceil-1} c_{2k+1} r^{(2k+1)\pi/(2\pi-\alpha)}.
    \end{align*}
\end{lemma}

\section{Subtraction technique}\label{app:subtraction_constant}
During the Nystr\"om discretization of the conical operator, we apply a subtraction technique to modify the jump relation near the conical corners~\cite{atkinson1997numerical}.

Let $G \subset \mathbb S^2$ denote the spherical cross-section of the cone, and $\partial G$ be parametrized by the curve $\bt(s)$. We still use the notation
\[
\chi(s,s')=\bt(s)\cdot \bt(s'),\quad h(s,s')=\bt(s)\cdot \bn(s'),
\]
where $\bn(s')$ is the unit normal vector of the conical surface at the source $\bt(s')$. Taking the limit $\mu\to0$ in the angular kernel gives
\begin{equation}\label{app:eq:K0_kernel}
    K_0(s,s')
    =
    \frac{1}{4\pi}
    \frac{h(s,s')}{1-\chi(s,s')}.
\end{equation}

Consider the truncated cone $\Omega_1=\{r\omega:\ 0<r<1,\ \omega\in G\}$, whose boundary consists of the conical side $S_1=\{r\bt(s):\ 0<r<1,\ s\in\partial G\}$ and the spherical cap $C_1=\{\omega:\omega\in G\}$. For the constant double-layer density on the closed surface $\partial\Omega_1$, the principal-value double-layer operator satisfies
\begin{equation}\label{app:eq:closed_dlp_constant}
    \mathcal D_{\partial\Omega_1}[1](\bx)=-\frac12,
    \qquad \bx\in \partial\Omega_1.
\end{equation}
For a target point $\bx=r_x\bt(s), 0<r_x<1$ on the conical side, this gives
\[
\mathcal D_{S_1}[1](\bx)+\mathcal D_{C_1}[1](\bx)=-\frac12.
\]

We first analyze the side contribution.
Indeed, for a source point $\by=r_y\bt(s')$, we have
\[
(\bx-\by)\cdot \bn(s')=r_x h(s,s').
\]
Therefore, we can obtain
\begin{align}
    \mathcal D_{S_1}[1](r_x\bt(s))
    &=
    \frac{1}{4\pi}
    \int_{\partial G}
    h(s,s')
    \int_0^1
    \frac{r_x r_y}
    {(r_x^2+r_y^2-2r_xr_y\chi(s,s'))^{3/2}}
    \,dr_y\,ds'. \label{app:eq:side_contribution}
\end{align}
Changing variables $r_y=r_x t$ gives
\[
    \int_0^1
    \frac{r_x r_y}
    {(r_x^2+r_y^2-2r_xr_y\chi)^{3/2}}
    \,dr_y
    =
    \int_0^{1/r_x}
    \frac{t}{(1+t^2-2\chi t)^{3/2}}\,dt.
\]
Passing to the vertex limit $r_x\to0$ yields
\[
    \int_0^{1/r_x}
    \frac{t}{(1+t^2-2\chi t)^{3/2}}\,dt
    \longrightarrow
    \int_0^\infty
    \frac{t}{(1+t^2-2\chi t)^{3/2}}\,dt
    =
    \frac{1}{1-\chi}.
\]
Hence the leading angular part of the side contribution is
\[
    \lim_{r_x\to0}\mathcal D_{S_1}[1](r_x\bt(s))
    =
    \frac{1}{4\pi}
    \int_{\partial G}
    \frac{h(s,s')}{1-\chi(s,s')}\,ds'
    =
    \mathcal K_0[1](s).
\]

It remains to evaluate the spherical cap contribution. For $\by=\omega\in C_1$, the outward normal is $\bn_{\by}=\omega$ and $dS_{\by}=d\omega$. Therefore,
\[
    \mathcal D_{C_1}[1](r_x\bt(s))
    =
    \frac{1}{4\pi}
    \int_G
    \frac{(r_x\bt(s)-\omega)\cdot\omega}
    {|r_x\bt(s)-\omega|^3}
    \,d\omega.
\]
Letting $r_x\to0$, we obtain
\[
    \lim_{r_x\to0}\mathcal D_{C_1}[1](r_x\bt(s))
    =
    -\frac{1}{4\pi}\int_G d\omega
    =
    -\frac{|G|}{4\pi},
\]
where $|G|$ denotes the spherical area of $G$. Taking the limit $r_x\to0$ in \eqref{app:eq:closed_dlp_constant} hence gives
\begin{equation}\label{app:eq:K0_constant}
\frac12+\mathcal K_0[1](s)=\frac{|G|}{4\pi}.
\end{equation}

In the numerical experiments, the corresponding normalized solid-angle constant is defined by $c_{\rm ref} = |G|/4\pi$. For the droplet-type examples in the numerical section, the spherical boundary
is generated by
\[
    \zeta(s)=f(s)\exp(i\theta(s-\tfrac12)),\quad 0\le s\le 1,
\]
through the projection
\[
    \bt(s)=\frac{(\Re\zeta(s),\Im\zeta(s),1)^T}{\sqrt{1+|\zeta(s)|^2}}.
\]
Recall the standard solid-angle line integral
\[
    |G|=\left|\int_{\partial G}(1-z)\,d\phi\right|,
\]
where $\phi$ denotes the azimuthal angle and $z=\cos\vartheta$ the polar height
along $\partial G$. Noting that $z(s)=1/\sqrt{1+f(s)^2}$ and $d\phi=\theta\,ds$ for the droplet parametrization, we obtain
\begin{equation}\label{app:eq:cref_1d}
    c_{\rm ref} =
    \frac{\theta}{4\pi}
    \int_0^1
    \left(
    1-\frac{1}{\sqrt{1+f(s)^2}}
    \right)\,ds,
\end{equation}
which can be evaluated numerically for any given $f$ and $\theta$.

\begingroup
\raggedright
\sloppy
\emergencystretch=1em
\bibliographystyle{elsarticle-num}
\bibliography{references}
\endgroup

\end{document}